\documentclass[12pt]{article}
\usepackage{amssymb, enumerate}
\usepackage{latexsym}
\usepackage{url}

\usepackage{amsmath,amssymb}

\newtheorem{theorem}{Theorem}[section]
\newtheorem{lemma}[theorem]{Lemma}
\newtheorem{corollary}[theorem]{Corollary}
\newtheorem{proposition}[theorem]{Proposition}

\newtheorem{conjecture}[theorem]{Conjecture}

\newtheorem{definition}[theorem]{Definition}

\newenvironment{proof}{\normalsize {\sc Proof}.}{{\hfill $\Box$%
 \hskip - \parfillskip\bigskip}}

\newcommand{\Syl}{\mathop{\rm Syl}\nolimits}
\newcommand{\Irr}{\mathop{\rm Irr}\nolimits}

\newcommand{\Aut}{\mathop{\rm Aut}\nolimits}
\newcommand{\Out}{\mathop{\rm Out}\nolimits}
\def\CO{{\mathcal{O}}}

\newcommand{\bG} {{\bf G}}
\newcommand{\bH} {{\bf H}}
\newcommand{\bL} {{\bf L}}
\newcommand{\bM} {{\bf M}}
\newcommand{\bT} {{\bf T}}
\newcommand{\bU} {{\bf U}}
\newcommand{\bX} {{\bf X}}
\newcommand{\bY} {{\bf Y}}

\def\bigcp{\mathop{\mathchoice % central product, large operator
 {\hbox{\sf\Large\lower 0.1\baselineskip\hbox{Y}}}%
 {\hbox{\sf\large\lower 0.1\baselineskip\hbox{Y}}}%
 {\hbox{\sf\normalsize\lower 0.1\baselineskip\hbox{Y}}}%
 {\hbox{\sf\tiny\lower 0.1\baselineskip\hbox{Y}}}%
}}
\def\bigtimes{\mathop{\mathchoice % direct product, large operator
 {\hbox{\sf\Large\lower 0.1\baselineskip\hbox{X}}}%
 {\hbox{\sf\large\lower 0.1\baselineskip\hbox{X}}}%
 {\hbox{\sf\normalsize\lower 0.1\baselineskip\hbox{X}}}%
 {\hbox{\sf\tiny\lower 0.1\baselineskip\hbox{X}}}%
}}

\def\Sym(#1){\mathop{\rm Sym}(#1)}
\def\Sym(#1){S_{#1}}
\def\diag(#1){\mathop{\rm diag}(#1)}

\newenvironment{enumerate*}{%
 \begin{enumerate}%
 }%
 {\end{enumerate}}

\begin{document}

\title{$2$-blocks with abelian defect groups}

\author{Charles W. Eaton, Radha Kessar, Burkhard K\"{u}lshammer and\\ Benjamin Sambale}

\date{22nd May 2013}
\maketitle

%%%%%%%%%%%%%%%%%%%%%%%%%%%%%
%%%%%%%%%%%%%%%%%%%%%%%%%%%%%%%%%%%%%%%%%%%%%
%Abstract

\begin{abstract}
We give a classification, up to Morita equivalence, of $2$-blocks of quasi-simple groups with abelian defect groups. As a consequence, we show that Donovan's conjecture holds for elementary abelian $2$-groups, and that the entries of the Cartan matrices are bounded in terms of the defect for arbitrary abelian $2$-groups. We also show that a block with defect groups of the form $C_{2^m} \times C_{2^m}$ for $m \geq 2$ has one of two Morita equivalence types and hence is Morita equivalent to the Brauer correspondent block of the normaliser of a defect group. This completes the analysis of the Morita equivalence types of $2$-blocks with abelian defect groups of rank $2$, from which we conclude that Donovan's conjecture holds for such $2$-groups. A further application is the completion of the determination of the number of irreducible characters in a block with abelian defect groups of order $16$. The proof uses the classification of finite simple groups.
\end{abstract}

%%%%%%%%%%%%%%%%%%%%%%%%%%%%%%%%%%%%%%%%%%%%%%%%%%%%%%%%%%%%%%%%%%
%Introduction

\section{Introduction}

Let $k$ be an algebraically closed field of prime characteristic $\ell$, and let $\CO$ be a discrete valuation ring with residue field $k$. Let $G$
be a finite group, and let $B$ be a block of the group algebra $\CO G$ with
defect group $D$. Assume that $\CO$ contains a primitive $|D|$-th root of unity.

The motivation for this paper is Donovan's conjecture, which states that for a fixed $\ell$-group $D$, there should only be a finite number of Morita equivalence classes of blocks with defect groups isomorphic to $D$. This conjecture is stated for blocks with respect to $k$, but it is also expected to hold for blocks with respect to $\CO$ (however, there are key results used in reduction arguments for the conjecture that at present are only known for $k$).

The main result   is that every $2$-block of a quasi-simple group with an abelian defect group is either one of a short list of exceptional cases or is Morita equivalent  over  $\CO$ to either a block covered by a nilpotent block, or to a tensor product of a nilpotent block with a block with Klein $4$-defect groups (see Theorem \ref{ab-2-str}). Blocks covered by nilpotent blocks are treated in~\cite{pu10}, where it is shown that they are Morita equivalent to their Brauer correspondent in the normaliser of a defect group. The Morita equivalences are achieved through the Bonnaf\'e-Rouquier correspondence.

The above may be viewed as being in the spirit of Walter's classification of simple groups with abelian Sylow $2$-subgroups, and it is hoped that it will eventually be used to tackle Donovan's conjecture for $2$-blocks with arbitrary abelian defect groups. We begin here with some cases which we may tackle with tools already at our disposal.

We prove that Donovan's conjecture holds for $2$-blocks with elementary abelian defect groups (Theorem \ref{elabDonovan}). In this case we are restricted to blocks defined over $k$ because we are reliant on the results of~\cite{ku95}.

A conjecture of Brauer (from his Problem $22$) represents a weak version of Donovan's conjecture. It states that for a given $\ell$-group $D$, there is a bound on the entries of the Cartan matrix of a block with defect groups isomorphic to $D$. This conjecture has been reduced to quasi-simple groups by D\"uvel in~\cite{du04}. We use our main result to show that the conjecture holds for all abelian $2$-groups (Theorem  \ref{weak-donovan}).

We also consider the case that $D$ is abelian of rank $2$, so that $D$ is isomorphic to a direct
product $C_{2^m} \times C_{2^n}$ of two cyclic subgroups $C_{2^m}$ and
$C_{2^n}$. If $m \neq n$, then every block with defect group $D$ is necessarily nilpotent. The case $m=n=1$ is the Klein $4$-case, which is already well known by~\cite{er82}. We show the following:

\begin{theorem}
\label{homocyclictheorem} Let $G$ be a finite group and $B$ a block of $\CO G$ with defect group $D$. Suppose that $D \cong C_{2^m} \times C_{2^m}$ for some $m \geq 2$. Then $B$ is Morita equivalent to either $\CO D$ or $\CO (D \rtimes C_3)$.
\end{theorem}

By the remarks above, this completes the analysis of $2$-blocks with abelian defect groups of rank $2$, including the verification of Donovan's conjecture for these groups.

We are also able to show that Donovan's conjecture holds for groups of the form $C_{2^m} \times C_{2^m} \times C_2$ for $m \geq 3$ (Theorem \ref{2m2m2}).

Finally, we complete the determination of the number of irreducible characters and irreducible Brauer characters in a block with abelian defect groups of order $16$ (Theorem \ref {eldef16}). This was completed modulo one case in~\cite{ks13}.

The structure of the paper is as follows: In Section \ref{background} we collect together some results which we will use later. In Section \ref{Leviabelian2Sylow} we give an analysis of finite groups of quotients of Levi subgroups with abelian Sylow $2$-subgroups which will be important in later arguments.

As mentioned above, the main result uses the Bonnaf\'e-Rouquier Morita equivalence. However, this equivalence only applies to finite groups of Lie type in the strict sense. In Section \ref{br} we show that the Bonnaf\'e-Rouquier Morita equivalence induces Morita equivalences on certain quotient groups, so that it applies to the associated simple group.

Section \ref{reductive groups} contains the analysis of $2$-blocks of finite groups of Lie type defined over fields of odd characteristic. The structure is presented in more detail than in the statement of the main result. In Section~\ref{main theorem} we prove the main theorem. In Section~\ref{homocyclic} we consider blocks with homocyclic defect groups and prove Theorem \ref{homocyclictheorem}. We prove Donovan's conjecture for elementary abelian $2$-groups in Section~\ref{elabsection}.

In Section \ref{weakdonovansection} we prove that the weak version of Donovan's conjecture holds for abelian $2$-groups. In Section \ref{elabinvariants} we treat blocks with elementary abelian defect groups of order $16$ and inertial index $15$, so completing the work of~\cite{ks13}. Finally in Section~\ref{homocyclicplusabit} we prove that Donovan's conjecture holds for groups $C_{2^m} \times C_{2^m} \times C_2$ for $m \geq 3$.

%%%%%%%%%%%%%%%%%%%%%%%%%%%%%%%%%%%%%%%%%%%%%%

\section{Background material}
\label{background}

For $G$ a   finite group, we  will use the   term  ``$\ell$-block of  $G$" to denote either a block of $\CO G $  or $ kG$;  the base ring will be specified   as needed.

We will make frequent use of the following, which apply for any prime $\ell$:

\begin{proposition}[\cite{Wa94}]
\label{watanabe}
Let $B$ be an $\ell$-block of a finite group $G$ and let $Z \leq Z(G)$ be an $\ell$-subgroup. Let $\bar B$ be the unique block of $G/Z$ corresponding to $B$. Then $B$ is nilpotent if and only if $\bar B$ is nilpotent.
\end{proposition}

\begin{proposition}[\cite{kp}]
\label{kp}
Let $G$ be a finite group and $N \lhd G$. Let $B$ be a block of $\CO G$ with defect group $D$ covering a nilpotent block $b$ of $\CO N$ with defect group $D \cap N$ and stabiliser $H$. 
Then there is a finite group $L$ and $M \lhd L$ such that (i) $M \cong D \cap N$, (ii) $L/M \cong H/N$, 
(iii) there is a subgroup $D_L$ of $L$ with $D_L \cong D$ and $ M \leq D_L$, and (iv) there is
a central extension $\tilde{L}$ of $L$ by an $\ell'$-group, and  a
 block $\tilde{B}$ of  $\CO \tilde L$   which is Morita equivalent to $B$ and has defect group $\tilde{D} \cong D$.
\end{proposition}

We will show that many blocks of finite groups of Lie type are covered by a nilpotent block. These blocks have very nice properties, as shown by Puig.

\begin{definition} We say that a block $B$ of a finite group $G$ is nilpotent-covered if there exists a group $\tilde G$ containing $G$ as a normal subgroup, and a nilpotent block $\tilde B$ of $\tilde G$ covering $B$.
\end{definition}

\begin{proposition}[\cite{pu10}]
\label{puig}
Every nilpotent-covered block $B$ of $ \CO G$ with defect group $D$ is Morita equivalent to its Brauer correspondent in $\CO N_G(D)$.
%Let $G$ be a finite group and $N \lhd G$. Let $b$ be a block of $N$ with defect group $D$ and suppose that $b$ is covered by a nilpotent block $B$ of $G$. Then $b$ is Morita equivalent to the unique block of $N_N(D)$ with Brauer correspondent $b$.
\end{proposition}

\begin{proof}
This is part of Corollary 4.3 of~\cite{pu10}.
\end{proof}

\begin{lemma}
\label{nilcoveredandquotients}
Let $G$ be a finite group and $N \lhd G$ with $Z(N) \leq Z(G)$. Let $\bar b$ be a block of $N/Z(N)$, and let $b$ be the unique block of $N$ corresponding to $\bar b$ with $O_{\ell'}(Z(N))$ in its kernel. Then $b$ is covered by a nilpotent block of $G$ if and only if $\bar b$ is covered by a nilpotent block of $G/Z(N)$.
\end{lemma}

\begin{proof}
This is an almost immediate corollary of Proposition \ref{watanabe}.
\end{proof}

The main result of~\cite{kk96} applies particularly well to blocks with elementary abelian defect groups:

\begin{proposition}
\label{KK} Let $G$ be a finite group  and let  $B$ be a block of $kG$ with elementary abelian defect group $D$ and suppose $N \unlhd G$ with $G=ND$. If $B$ covers a $G$-stable block $b$ of $kN$, then  there is an elementary abelian $\ell$-group $Q$  such that $B$ is Morita equivalent to a block $C$  of  $k(N \times Q) $  with  defect group $(D \cap N) \times Q \cong D$.
\end{proposition}

\begin{proof}
We may write $D=(D \cap N) \times Q$ for some $Q \leq D$. By the main result of~\cite{kk96}, $B \cong kQ \otimes_k b$ as $k$-algebras. Observe that $kQ \otimes_k b$ is a block of $N \times Q$ with defect group $D=(D \cap N) \times Q$.
\end{proof}

For dealing with finite simple groups, a powerful reduction is provided by a theorem of Bonnaf\'e and Rouquier~\cite{BonRou}. However, to effectively use the Bonnaf\'e-Rouquier results, we will need also to have a version
for simple Chevalley groups which are not groups of Lie type (the main cases in question are the simple groups of type $E_7$, where the finite group of Lie type has either a centre of order two or a normal subgroup of index two). This will be done in Section~\ref{br} using the following well known result. A proof is given for the convenience of the reader.

\begin{lemma}
\label{moritaquotient}
Let $G$ and $H$ be finite groups, $b$ and $c$ be block  idempotents of $\CO G$ and $\CO H$ respectively. Let $M$ be an $(\CO Gb, \CO Hc)$-bimodule, and let
 $Z$ be an $\ell$-group embedded as a central subgroup of both
$G$ and $H$. Let $\bar G =G/Z $, $\bar H = H/Z$, and
 let $\bar b$ (respectively $\bar c $) be the image of $b$ (respectively $c$) in $\CO G/Z$ (respectively $\CO H/Z $) under the canonical surjection $\CO G \rightarrow \CO G/Z$ (respectively $\CO H \rightarrow \CO H/Z$).

 Suppose that $ zm = mz $ for all $z \in Z$ and all $m \in M$.
 Then $ \CO \otimes _{\CO Z} M $ is an $(\CO\bar G\bar b, \CO \bar H \bar c)$-bimodule via $ \bar g ( 1\otimes m) \bar h = 1\otimes gmh $, $g\in G, h\in H, m\in M$.
If $M \otimes_{\CO Hc}-$ induces a Morita equivalence between $\CO Hc $ and $\CO Gb$, then $(\CO \otimes_{\CO Z} M )\otimes_{\CO \bar H \bar c} -$ induces a Morita equivalence between $\CO \bar H \bar c $ and $\CO \bar G \bar b$.
\end{lemma}

\begin{proof} The first assertion is straightforward. Since $zm^* = m^*z $ for all $z\in Z$, $m^* \in M^* $ we have similarly that
$\CO\otimes_{\CO Z} M^*$ is an
$(\CO\bar H\bar c, \CO \bar G\bar b)$-bimodule via $ \bar h( 1\otimes m^*) \bar g= 1\otimes hm^*g $, $g\in G, h\in H, m^*\in M^*$.
We also have that in $M \otimes_{\CO Hc} M^* $, $z (m \otimes m^*)= (m \otimes m^*) z $ and in $M^*\otimes_{\CO G b} M $, $z (m^* \otimes m)=(m^* \otimes m) z $ for all $z\in Z$, $m\in M$, $m^*\in M^*$.
Further, $ (\CO \otimes _{\CO Z} M ) \otimes _{\CO \bar H\bar c} (\CO \otimes _{\CO Z} M^*) \cong \CO \otimes_{\CO Z}(M\otimes_{\CO Hc} M^*) $ as $(\CO \bar G \bar b, \CO \bar H \bar c)$-bimodule via $1\otimes m \otimes 1 \otimes m^* \rightarrow 1\otimes m\otimes m^* $, for $m\in M$, $m^* \in M^*$
and similarly, $ (\CO \otimes _{\CO Z} M^* ) \otimes _{\CO \bar G\bar b} (\CO \otimes _{\CO Z} M ) \cong \CO \otimes_{\CO Z}(M^*\otimes_{\CO Gb} M) $
as an $(\CO \bar H \bar c, \CO \bar G \bar b)$-bimodule.
The result follows since $ \CO \otimes _{\CO Z} \CO Gb \cong \CO \bar G \bar b $ as an $(\CO \bar G\bar b , \CO \bar G \bar b )$-bimodule and $ \CO \otimes _{\CO Z} \CO Hc \cong \CO \bar H \bar c $ as an $(\CO \bar H\bar c , \CO \bar H \bar c )$-bimodule.
\end{proof}

This is easily extended to the following:

\begin{lemma}
\label{bimodswithZinkernel}
Let $G_1,\ldots,G_t$ and $H_1,\ldots,H_t$ be finite groups. Let $b_i$ be a block  idempotent of $\CO G_i$ and $c_i$ be a block  idempotent of $\CO H_i$ for $i=1,\ldots,t$.
Write $G=G_1 \times \cdots \times G_t$ and $H=H_1 \times \cdots \times H_t$. Write $b=b_1\cdots b_t$ and $c=c_1 \cdots c_t$, so that $b$ is a block idempotent of $\CO G$ and $c$ is a block  idempotent  of $\CO H$. Let $M_i$ be an $(\CO Gb_i,\CO Hc_i)$-bimodule and let $Z_i$ be an $\ell$-group embedded as a central subgroup of both $G_i$ and $H_i$. Write $\bar{G}_i=G_i/Z_i$ and $\bar{H}_i=H_i/Z_i$. Write $M=M_1 \otimes_\CO \cdots \otimes_\CO M_t$, an $(\CO Gb,\CO Hc)$-bimodule.

Suppose that $ z_im_i = m_iz_i $ for all $z_i \in Z_i$ and all $m_i \in M_i$. Let $Z \leq Z_1 \times \cdots \times Z_t$, and write $\bar{G}=G/Z$ and $\bar{H}=H/Z$. Let $\bar{b}$ (resp. $\bar{c}$) be the block of $\bar{G}$ (resp. $\bar{H}$) corresponding to $b$ (resp. $c$).
Then $ \CO \otimes _{\CO Z} M $ is an $(\CO\bar G\bar b, \CO\bar H \bar c)$-bimodule via $ \bar g ( 1\otimes m) \bar h = 1\otimes gmh $, $g\in G, h\in H, m\in M$.
If $M_i \otimes_{\CO H_ic_i}-$ induces a Morita equivalence between $\CO H_ic_i $ and $\CO G_ib_i$ for each $i$, then $(\CO \otimes_{\CO Z} M )\otimes_{\CO \bar H \bar c} -$ induces a Morita equivalence between $\CO \bar H \bar c $ and $\CO \bar G \bar b$.
\end{lemma}

%%%%%%%%%%%%%%%%%%%%%%%%%%%%%%%%%%%%%%%%%%%%%%%

\section{On Levi subgroups with abelian Sylow $2$-subgroups}
\label{Leviabelian2Sylow}
Let ${\mathbb F}$ be an algebraically closed field of characteristic $p> 0$ and let $\bf G$ be a connected reductive group
defined over ${\mathbb F}$
with a Steinberg endomorphism $F : {\bf G} \to {\bf G}$, and $G={\bf G}^F$ the finite group of fixed points.
Define $q$ so that $\mathbb{F}_q$ is the field of definition of $G$.

Recall that $\bG =Z^{\circ} (\bG) [\bG, \bG] $, where $Z^{\circ} (\bG)$ is the connected centre of $\bG$, and that the derived subgroup $[\bG, \bG]$ of $\bG$ is a semi-simple group, that is $[\bG, \bG]$ is a commuting product of simple groups, called the components of $\bG$. We assume throughout this section that the fixed point subgroup of no $F$-orbit of components of $\bG$ is isomorphic to a Suzuki or Ree group.

\begin{lemma}\label{abelian-cyclic-quotient} Suppose that $p$ is odd and $[\bG, \bG] $ is simply connected (that is $[\bG, \bG] $ is a direct product of its components, each of which is simply connected).
%Suppose that $ Z \leq Z(G) $ is a $2$-group such that $G/Z $ has abelian Sylow $2$-subgroups.
\begin{enumerate}[(i)]
\item If $G$ has abelian Sylow $2$-subgroups, then $\bG$ is a torus.
\item Suppose that $G$ has non-abelian Sylow $2$-subgroups, but for some central $2$-subgroup $Z$ of $G$, $G/ Z$ has abelian Sylow $2$-subgroups. Then all components of $\bG$ are of type $A_1$. Further, if
$\{\bX_1, \bX_2, \ldots, \bX_r\} $ is an $F$-orbit of components of $[\bG, \bG] $, then $q^r \equiv \pm 3\pmod{8}$, $Z \cap (\prod_{1\leq i\leq r } \bX_i)^F \ne 1 $
and $ Z^{\circ} (\bG)^F \cap (\prod_{1\leq i\leq r } \bX_i)^F =1$.
\end{enumerate}
\end{lemma}

\begin{proof} The first statement is well-known. We prove (ii). Let $ \{\bX_1, \ldots, \bX_s \} $ be the set of components of $[\bG, \bG] $ and let $\{ \bX_1, \ldots ,\bX_r\}$, $ 1\leq r \leq s $ be an $F$-orbit of components such that $ F(\bX_i) = \bX_{i+1} $ for $1\leq i\leq r-1 $ and $F(\bX_r)=\bX_1 $.
Then $(\prod_i \bX_i)^F \cong \bX_1^{F^r}$, and denoting by $Z'$ the image of $ Z\cap (\prod_i \bX_i)^F $ under this isomorphism,
$ \bX_1^{F^r}/Z' $ has abelian Sylow $2$-subgroups. It follows that $\bX_1 $ is of type $A_1 $,
$\bX_1^{F^r}\cong SL_2(q^r)$ with $q^r \equiv \pm 3\pmod{8}$ and that $Z'$ is the unique central subgroup of order $2$ of $\bX_1^{F^r} $.
In other words, $Z \cap (\prod_{1\leq i\leq r } \bX_i)^F$ is generated by $ (z_1, F(z_1), \ldots, F^{r-1}(z_1 )) =:\zeta $ where $ z_1$ is the unique involution in the centre of $ \bX_1 $.
It remains only to show that $ \zeta \notin Z^{\circ}(\bG) $.
Suppose the contrary.
The multiplication map $ \mu : Z^{\circ} (\bG) \times [\bG, \bG] \to \bG $ is surjective with kernel $\Delta (Z^{\circ} (\bG) \cap [\bG, \bG ] )$, where $\Delta (Z^{\circ} (\bG) \cap [\bG, \bG ] )$ is the diagonally embedded copy of $Z^{\circ} (\bG) \cap [\bG, \bG ] $ in
$Z^{\circ} (\bG) \times [\bG, \bG] $.
Let $ A$ denote the inverse image under $ \mu $ of $G$. So,
\[A =\{ ( u, g) \, : \, u\in Z^{\circ}({\bG}), g \in [{ \bG}, {\bG}] \, \text{ such that } \, u^{-1} F(u) = g F(g^{-1}) \}.\]
Let $\tau_1: A \to \bX _1$ be the (restriction to $ A$ of the) projection of
 $Z^{\circ} (\bG) \times [\bG, \bG]= Z^{\circ} (\bG) \times \bX_1 \times \cdots \times \bX_s $ onto $\bX_1$. Since
$ A$ contains $Z^{\circ} (\bG)^F \times [\bG, \bG]^F $, $ \bX_1^{F^r} \leq \tau_1(A)$.

By the Lang-Steinberg theorem, there exist $2$-elements $u\in Z^{\circ} (\bG) $ and $ g\in [\bG, \bG]$ such that
$ u^{-1}F(u) = \zeta= gF(g^{-1})$. In particular, $ (u, g ) \in A $. Write $g=(x_1, \ldots, x_s) $, $x_i \in \bX_i $.
The equation $ \zeta= gF(g^{-1})$
implies that $ x_1= (z_1)^r F^r(x_1)$. Since $q^r \equiv \pm 3\pmod{8}$, $r$ is odd, hence
$\tau_1(u, g) =x_1 \notin \bX_1 ^{F^r} $. Consequently, a Sylow $2$-subgroup of $\tau_1(A)$ properly contains a Sylow $2$-subgroup of $\bX_1^{F^r}$ and it follows that the Sylow $2$-subgroups of $\tau_1(A) $ are non-abelian of order at least $16$. Since $\bX_1 $ is of type $A_1$, any finite $2$-subgroup of $\bX_1 $ is cyclic or quaternion. Hence the Sylow $2$-subgroups of $\tau_1 (A) $ are quaternion of order at least $16$.

Let $S$ be a Sylow $2$-subgroup of $A$. By hypothesis, $[\mu (S), \mu (S)] \leq Z $ and the inverse image under $\mu $ of $Z$ is a central subgroup of
$Z^{\circ} (\bG) \times [\bG, \bG] $ (note that $Z^{\circ}( \bG) \times Z([\bG, \bG])$ is the full inverse image under $\mu $ of $Z(\bG)$ and $Z \leq Z(\bG)$). Hence, $ [S, S]\leq Z(S)$ from which it follows that $[\tau_1(S), \tau_1(S)] \leq \tau_1(Z(S)) \leq Z(\tau_1(S))$, a contradiction.
\end{proof}

We now fix a maximal torus $\bT$ of $\bG $ and assume that ${\mathbb F} $ has odd characteristic. Let $X(\bT)$ be the group of rational characters of $\bT$, $ Y(\bT)$ the group of one-parameter subgroups of $\bG$, and $<, > : X(\bT) \times Y(\bT) \rightarrow {\mathbb Z} $ the canonical exact pairing. Let $\Phi \subset X(\bT)$ be the set of roots of $\bG$ with respect  to $\bT $, $\Phi^{\vee} $ the  corresponding set of coroots and $ I \subset \Phi $ a set of fundamental roots
corresponding to a Borel subgroup of $ \bG $ containing $\bT$. For each $\alpha \in \Phi $, denote by $\bU_{\alpha}$ the corresponding root subgroup of $\bG$ and let $\phi_{\alpha} : SL_2({\mathbb F}) \to \langle \bU_{\alpha}, \bU_{-\alpha} \rangle $ be a surjective homomorphism such that the image of the group of upper triangular matrices is $\bU_{\alpha}$ and the image of the group of lower triangular matrices is $\bU_{-\alpha}$ (see \cite[Prop.~0.44]{DiMi}). If $\beta \in \Phi $ and $ a \in {\mathbb F} ^{\times} $, then $\beta( \phi_{\alpha}( \begin{smallmatrix} a &0 \\ 0&a^{-1} \end{smallmatrix} ))= a^{<\beta, \alpha^{\vee}>} $.

For $J \subseteq I $, let $\Phi_J $ be the set of roots which are in the subspace of ${\mathbb R} \otimes X(\bT)$ generated by $J$. 
The group $\bM_J:= \langle \bU_{\alpha}, \bU_{-\alpha} \, : \, \alpha \in \Phi_J \rangle $ is the derived subgroup of the Levi subgroup $\bL_J:= \langle \bT, \bM_J\rangle $ of $\bG$ and any Levi subgroup of $\bG $ is conjugate to $\bL_J $ for some $ J \subseteq I$. All components of $\bM_J $ are of type $A_1 $ if and only if
for each $\alpha, \beta \in J $ with $\alpha \ne \beta $, $<\alpha, \beta^{\vee} > =0 $ and in this case,
$\bM_J = \prod_{\alpha \in J} \langle \bU_{\alpha}, \bU_{-\alpha} \rangle $. Further, if $\bG$ is simply connected, then the product $ \prod_{\alpha \in J} \langle \bU_{\alpha}, \bU_{-\alpha} \rangle $ is direct and for all $\alpha \in J$,
$\phi_{\alpha} $ is an isomorphism. In particular, if $\bG $ is simply connected and $char(\mathbb F) $ is odd, then $z_{\alpha}:= \phi_{\alpha} ( \begin{smallmatrix} -1 &0 \\ 0& -1 \end{smallmatrix}) $ is the
 unique central element of order $2$ in $\langle \bU_{\alpha}, \bU_{-\alpha} \rangle $.

\begin{lemma}\label{dynkin-class} Suppose that $char(\mathbb F) $ is odd and that $\bG$ is simple and simply connected. Let $ \emptyset \ne J\subseteq I $ such that if $\alpha, \beta\in J$ with $\alpha \ne \beta $, then $<\alpha, \beta^{\vee}> =0 $. Let $ z:= \prod_{\alpha\in J } z_{\alpha}$ and suppose that $z \in Z(\bG )$.
\begin{enumerate}[(i)]
\item Suppose that $\bG $ is of type $A_n $, $ n\geq 1 $. Let $I =\{\alpha_1, \ldots, \alpha_n \}$, where $\alpha_i, \alpha_{i+1} $, $1\leq i \leq n-1 $ are consecutive nodes in the corresponding Dynkin diagram. Then $n$ is odd and
$ J =\{\alpha_1, \alpha_3, \alpha_5, \ldots, \alpha_{n-2}, \alpha_n \} $.
\item Suppose that $\bG $ is of type $B_n $, $n \geq 2 $. Let $I =\{\alpha_1, \ldots, \alpha_n \}$, where $\alpha_i, \alpha_{i+1} $, $1\leq i \leq n-1 $ are consecutive nodes in the corresponding Dynkin diagram and there is a double arrow from $ \alpha_{n-1} $ to $\alpha_n $. Then $ J =\{ \alpha_n \} $.
\item Suppose that $\bG $ is of type $C_n $, $n \geq 3 $. Let $I =\{\alpha_1, \ldots, \alpha_n \}$, where $\alpha_i, \alpha_{i+1} $, $1\leq i \leq n-1 $ are consecutive nodes in the corresponding Dynkin diagram and there is a double arrow from $ \alpha_{n} $ to $\alpha_{n-1} $. If $n $ is even, then $ J = \{ \alpha_1, \alpha_3, \ldots, \alpha_{n-3}, \alpha_{n-1} \} $.
 If $n $ is odd, then $ J = \{ \alpha_1, \alpha_3, \ldots, \alpha_{n-2}, \alpha_{n} \} $.
\item Suppose that $\bG $ is of type $D_n $, $n \geq 4 $ . Let $I =\{\alpha_1, \ldots, \alpha_n \}$, where $\alpha_i, \alpha_{i+1} $, $1\leq i \leq n-2 $ are consecutive nodes and $\alpha_ n$ and $\alpha_{n-2} $ are connected.
Then either $ J =\{ \alpha_{n-1}, \alpha_n \} $ or $n$ is even and $J$ is one of $ \{ \alpha_1, \alpha_3, \ldots, \alpha_{n-3}, \alpha_{n-1} \} $ or $ \{ \alpha_1, \alpha_3, \ldots, \alpha_{n-3}, \alpha_{n} \} $.
\item Suppose $\bG$ is of type $E_7 $. Let $I =\{\alpha_1, \ldots, \alpha_7\}$, such that $ I -\{ \alpha_2\} $ corresponds to the Dynkin diagram of type $A_6 $ and $ \alpha_2 $ is connected to $\alpha_5$. Then $ J =\{ \alpha_1, \alpha_2, \alpha_4 \} $.
\end{enumerate}
\end{lemma}

\begin{proof} Let $\alpha, \beta \in I$. Then as explained above, $\beta(z_{\alpha})= (-1 )^{< \beta, \alpha^{\vee} >}$.
By hypothesis, $z \in Z(\bG)$ whence $ \prod_{\alpha \in J } \beta(z_{\alpha})= \beta(z) =1 $ for all $ \beta\in \Phi $ (see
\cite[Prop~0.35]{DiMi}).
So, if $< \beta, \alpha^{\vee} >$ is odd (that is equals $-1$ or $-3$) and for all $\gamma \in J $ different from $\alpha $, $ < \beta, \gamma^{\vee} > $ is even (that is equals $0$ or $\pm 2 $) ,
then $\alpha \notin J $. We will systematically use this observation. Also note that since the elements of $J$ are pairwise orthogonal, $J$ does not contain any pair of consecutive nodes.

Suppose first that $\bG $ is of type $A_n$, $n\geq 1 $. Then, $<\alpha_i, \alpha_j^{\vee} > $ is odd if and only if $ j =i\pm 1$.
So, $\alpha _2 $ is the unique element of $I$ such that $<\alpha_1, \alpha_2^{\vee}> $ is odd. By the observation above $\alpha_2 \notin J $. We claim that $\alpha_1\in J $.
Indeed, suppose not and let $ i $ be the least integer such that $\alpha_i \in J$. Then $i \geq 3 $ and $ <\alpha_{i-1}, \alpha_j^{\vee} > $ is odd if and only if $ j =i-2 $ or $ i $.
Since $\alpha_{i-2} \notin J $, it follows that $ i\notin J$, a contradiction. Thus, $\alpha_1\in J$ and $\alpha_2 \notin J $. The result follows by repeating the argument.

Suppose that $\bG $ is of type $B_n $, $n \geq 2$. Then, $ <\alpha_i, \alpha_j^{\vee} > $ is odd if and only if $ 1\leq i, j \leq n-1 $ and $j = i\pm 1 $ or $ i=n$ and $j=n-1$.
In particular, $<\alpha_n, \alpha_j^{\vee} > $ is odd if and only if $j =n-1 $, hence $\alpha_{n-1}\notin J$. If $n=2 $, the result is proved.
Suppose that $n \geq 3 $. Then $<\alpha_{n-1}, \alpha_j^{\vee} > $ is odd if and only if $j=n-2 $, hence $\alpha_{n-2} \notin J $. Suppose that $n \geq 4 $ and let $ i $ be the greatest integer such that $ i\leq n- 3 $ and $ i \in J$. Then $<\alpha_{i+1}, \alpha_j^{\vee} > $ is odd if and only if $ j=i $ or $ j =i+2 $. By maximality of $i $, $ n-1 \geq i+2 \notin J $, hence $ i \notin J $, a contradiction. Thus, $ J =\{ \alpha_n \} $.

Suppose that $\bG $ is of type $C_n $, $n \geq 3$. Then, $ <\alpha_i, \alpha_j^{\vee} > $ is odd if and only if $ 1\leq i, j \leq n-1 $ and $j =i\pm 1 $ or $ i=n-1$ and $j=n$. Suppose first that $n-1 \in J $. Then $ n-3 \in J$ as $ <\alpha_{n-2}, \alpha_j^{\vee} > $ is odd if and only if $ j=n-1 $ or $n-3 $. Continuing like this, we get that $ n$ is even and $ J =\{ \alpha_1, \alpha_3, \ldots, \alpha_{n-1} \} $. Now suppose that $ \alpha_{n-1} \notin J $. We claim that $\alpha_n\in J$. Indeed, suppose not and let $ i $ be the greatest integer such that $\alpha_i \in J $. Then $ i \leq n-2 $, and $<\alpha_{i+1}, \alpha_j^{\vee}> $ is odd if and only if either $j=i $ or $ j= i+2 $. Since $i+2 \notin J $, $ i \notin J$, a contradiction. Thus, $ \alpha_n \in J $ from which it follows that $n-2 \in J$. Continuing, one obtains that $ n$ is odd and $ J =\{ \alpha_1, \alpha_3, \ldots, \alpha_{n} \} $.

Suppose that $\bG $ is of type $D_n $. Then, $\alpha_{n-2} \notin J $ since $<\alpha_{n}, \alpha_j^\vee > $ is odd if and only if $ j=n-2 $. Suppose first that $\alpha_{n-3} \notin J$. Then it follows that $\alpha_i \notin J $ for any $i \leq n- 3 $ whence $ J \subseteq \{\alpha_{n-1}, \alpha_n \}$ and consequently $ J = \{ \alpha_{n-1}, \alpha_n \} $. The case that $\alpha_{n-3} \in J $ leads to the conclusion that $ n $ is even and $ J$ is one of the two sets claimed.

Finally suppose that $\bG $ is of type $E_7 $. Then, $<\alpha_1, \alpha_j^{\vee}> $ is odd if and only if $ j =3 $, hence $\alpha_3 \notin J $. Also, $<\alpha_4, \alpha_j^{\vee} > $ is odd if and only if $j=3 $ or $j=5 $, hence
$\alpha_5 \notin J $. Since
 $<\alpha_6, \alpha_j^{\vee} > $ is odd if and only if $j=7 $ or $j=5 $, $\alpha_7 \notin J $. Since $<\alpha_7, \alpha_j^{\vee} > $ is odd if and only if $j=6 $, $\alpha_6\notin J $. Similarly, it follows that $\alpha_2 \in J$ if and only
if $\alpha_4 \in J$ if and only if $\alpha_1 \in J $. Hence $ J =\{ \alpha_1, \alpha_2, \alpha_4 \} $ as claimed.
\end{proof}

\begin{lemma}\label{dynkin-con} Keep the notation and hypothesis of Lemma \ref{dynkin-class}.
Let $\bL_J = \langle \bT, \bU_{\alpha}, \bU_{-\alpha}: \alpha \in J \rangle $ be the Levi subgroup corresponding to $J$.
\begin{enumerate}[(i)]
\item If $\bG$ is of type $B_n$, $n\geq 2 $, or $C_n $, $n \geq 3$ and $n$ even, then $Z(\bL_J)$ is connected.
\item If $\bG$ is of type $C_n $, $n$ odd, then the components of $\bL $ are not transitively permuted by $N_W(W_J)$.
\end{enumerate}
\end{lemma}

\begin{proof} (i) Let $P$ be the subgroup of $ X(\bT )$ generated by the $\alpha_i, $ $i \in J$. It suffices to show that
$ X(\bT)/P$ is torsion-free (see for instance \cite[Lemma 13.14]{DiMi}). Keep the labelling of the fundamental roots introduced in Lemma \ref{dynkin-class}. Let $q_i$, $1\leq i \leq n $ be the set of fundamental weights corresponding to $\Phi$, $\Phi^{\vee}$ and $I$. Thus $q_i $ are vectors in ${\mathbb R}\otimes X(\bT ) $ defined by
$ < q_i, \alpha_j^{\vee} > =\delta_{i,j} $, $1\leq i, j\leq n$.
Since $\bG$ is simply connected, $X(\bT )$ equals the subgroup of ${\mathbb R}\otimes X(\bT ) $ generated by the fundamental weights.

Suppose that $ \bG$ is of type $B_n$, $n \geq 2 $.
Then ${\mathbb R}\otimes X(\bT ) $ may be identified with an $n$-dimensional Euclidean space with an orthonormal basis $ e_1, \ldots,e_n $ such that under this identification,
\[ \alpha_1 = e_1-e_2, \ldots, \alpha_{n-1} = e_{n-1}- e_{n}, \alpha_n= e_n \]
and
\[ q_1= e_1, q_2= e_1 + e_2, \ldots, q_{n-1}= e_1 + \cdots + e_{n-1}, q_n = \frac{1}{2} (e_1 +e_2 +\cdots + e_n). \]
So, $X(\bT)$ is generated by $e_1, e_2, \ldots e_n, q_n $ and $e_1, e_2, \ldots, e_{n-1}, q_n $ is a basis of $X(\bT)$.
By Lemma~\ref{dynkin-class}, $P= {\mathbb Z} e_n $. So, $X(\bT) /P $ is free with basis $e_1, \ldots, e_{n-2}, q_n $.

Suppose that $ \bG$ is of type $C_n$, $n \geq 3 $. Then ${\mathbb R}\otimes X(\bT ) $ may be identified with an $n$-dimensional Euclidean space with an orthonormal basis $ e_1, \ldots, e_n $ such that under this identification,
\[ \alpha_1 = e_1-e_2, \ldots, \alpha_{n-1} = e_{n-1}- e_{n}, \alpha_n= 2e_n \]
and
\[ q_1= e_1, q_2= e_1 + e_2, \ldots, q_{n-1}= e_1 + \cdots + e_{n-1}, q_n = e_1 +e_2 +\cdots + e_n. \]
So, $e_1, e_2, \ldots, e_{n-1}, e_n $ is a basis of $X(\bT)$.

If $n$ is even, then by Lemma~\ref{dynkin-class},
\[P={\mathbb Z} (e_1-e_2) \oplus {\mathbb Z} (e_3-e_4) \oplus \cdots \oplus {\mathbb Z} (e_{n-1}-e_{n})\]
 and $X(\bT) /P $ is free with generators $ e_1 +P, e_2 +P, \ldots, e_{n-2} +P $.

(ii) Suppose that $\bG$ is of type $C_n $, $n$ odd. By Lemma \ref{dynkin-class}, $J=\{ \alpha_1, \alpha_3, \ldots, \alpha_{n-2}, \alpha_n\}$.
 Now, $\alpha_n= 2e_n $, $\alpha_i= e_i -e_{i+1} $ for $1\leq i \leq n-1 $ and for any $i$, $1\leq i \leq n $ and any $w\in W$, $\,^w (e_i) =\pm e_j $, hence $\alpha_n $ is not in the same $W$-orbit as $\alpha_{i}$ for any $i \leq n-1$.
\end{proof}

 We assume from now on that $\bT$ is $F$-stable. Recall that  to any $F$-stable Levi subgroup $\bL$ of $\bG$ is associated a pair $(J, w )$, where $ w \in W$, $ J \subset I $ such that $\,^{wF}J = J $ and $ \bL = \,^g \bL_J $ for some $g \in \bG$ with $g^{-1} F(g) \in N_{\bG}(\bT) $ whose image in $ W=N_{\bG}(\bT)/\bT $ is $w$. Moreover, $ \bL ^F \cong \bL_J ^{wF} $ (see \cite[Prop.~4.3]{DiMi}).

\begin{proposition}\label{2abelian-levi} Suppose that ${\mathbb F} $ has odd characteristic, $\bG$ is simple and simply connected and that $\bT$ is $F$-stable. Let $\bL $ be an $F$-stable non-toral Levi subgroup of $\bG$, $Z$ a central $2$-subgroup of $G$.
Suppose that the Sylow $2$-subgroups of $\bL^F/Z $ are abelian. Then,
\begin{enumerate} [(i)] \item $F$ has only one orbit on the set of components of $[\bL, \bL]$.
\item $[\bL, \bL]^F \cong SL_2 (q^t) $, $q \equiv \pm 3\pmod{8}$, where $t$ is odd and equals the number of components of $[\bL, \bL]$. Further, $ \bL^ F $ is a direct product of $[\bL, \bL]^F $ and $Z^{\circ}(\bL)^F$.
\end{enumerate}
Let $P$ be a Sylow $2$-subgroup of $\bL^F$ and write $P=P_0 \times P_1 $ where $P_0 \leq Z^{\circ}(\bL)^F $ and $P_1\leq [\bL, \bL]^F$.
Then $P_1$ is quaternion of order $8$, $ Z \cap P_1$ is the central subgroup of $[\bL, \bL]^F$ and $P_1/(Z\cap P_1)$ is a Klein $4$-group.
Moreover, one of the following holds
\begin{enumerate}[(a)]\item $\bG $ is of type $A_n $, $n+1= 2t$, $P_0 $ is trivial and $Z$ has order $2$. In particular, $P/Z \cong C_2 \times C_2 $.
\item $\bG $ is of type $D_n $, $n =2t$, $ G $ is of untwisted type $D_n(q) $ and either $Z$ is of order $2$ or a Klein $4$-group. If $ q\equiv 3\pmod{4}$, then $P_0 $ has order $2$ and $P/Z $ is elementary abelian of order $8$ or a Klein $4$-group, depending on whether $Z$ has order $2$ or $4$. If $ q\equiv 1\pmod{4}$, then $P_0 $ has order $4$ and $P/Z$ is isomorphic to $C_4 \times C_2 \times C_2 $ or is elementary abelian of order $8$, depending on whether $Z$ has order $2$ or $4$.
\item $\bG $ is of type $E_7$ and $t=3$.
\end{enumerate}
\end{proposition}

\begin{proof} Note that $|\bL^F|= |Z^{\circ}(\bL)^F| |[\bL, \bL]^F|$ . So, if $Z^{\circ} (\bL ) \cap [\bL, \bL]^F =1$, then
$ \bL ^F $ is a direct product of $Z^{\circ} (\bL ) ^F $ and $ [\bL, \bL]^F $. Thus, (ii) is a consequence of (i) and Lemma~\ref{abelian-cyclic-quotient}.

We prove (i). By Lemma~\ref{abelian-cyclic-quotient} (applied to $\bL$), $Z$ and hence $O_2(G) $ is non-trivial. Hence, $\bG$ is of type $A_n $, $B_n $, $C_n $, $D_n $ or $ E_7 $. Similarly we may also assume that $G$ is not of type $\,^3D_4 (q)$.
Let $ (J, w) $ be associated to $\bL $ as explained above.
The $F$-orbits of $\bL $ on the components of $\bL$ correspond to $wF$-orbits of $J$. Hence it suffices to prove that there is only one $wF$-orbit on $J$.
Let $J_1 \subset J $ be a $wF$-orbit of $J$ of size $t$. By Lemma \ref{abelian-cyclic-quotient},
$[\bL_{J_1}, \bL_{J_1}] ^{wF} \cong [\bL, \bL]^F \cong SL_2(q^t) $, where $ q^t \equiv \pm 3\pmod{8}$. So, $ t $ is odd. Also by Lemma~\ref{abelian-cyclic-quotient} and transport of structure, we have
$[\bL_{J_1}, \bL_{J_1}]^{wF} \cap Z \ne 1 $ (here note that any $\bG$-conjugate of $Z$ equals $Z$). So, in particular
$[\bL_{J_1}, \bL_{J_1}] \cap Z(\bG) \ne 1 $. We now apply Lemma \ref{dynkin-class} and Lemma~\ref{dynkin-con} to $J_1$ (note that by Lemma~\ref{abelian-cyclic-quotient}, all components of $\bL_J $ have rank $1$).

Suppose $\bG$ is of type $A_n $. Then by Lemma \ref{dynkin-class}, $n$ is odd and $J_1 = \{ \alpha_1, \alpha_3, \ldots, \alpha_n\} $. So,
$n+1 =2t$. It also follows that $J=J_1 $, since $J_1 $ is the only subset of $J$ with the required properties. Thus, all statements for type $A_n$ are proved except the assertion on $P_1 $ being trivial. For this, we use the order formulas for
$Z^{\circ}(\bL)^{wF}$ given in \cite{carter78}. By Proposition 7 and 8 of \cite{carter78}, we see that $|Z^{\circ} ( \bL)^F|= \frac{q^t -1}{q-1} $ if
$ G $ is untwisted and $|Z^{\circ} ( \bL)^F|=\frac{q^t +1}{q+1} $ if $G$ is twisted. Since $t$ is odd, in either case, we have that
$|Z^{\circ} ( \bL)^F|$ is odd. This proves the proposition for groups of type $A_n$.

Suppose $\bG$ is of type $B_n $. Then by Lemma \ref{dynkin-class}, $J=J_1=\{\alpha_{n} \} $. By Lemma \ref{dynkin-con}, $ Z^{\circ} (\bL_I) $ is connected, hence $ Z\leq Z^{\circ} (\bL_I) $. But by Lemma \ref{abelian-cyclic-quotient}
$Z \cap [\bL_J, \bL_J ] ^{wF} \ne 1 $ and $Z^{\circ} (\bL ) \cap [\bL_J, \bL_J ] ^{wF} = 1 $, a contradiction.

Suppose $\bG$ is of type $C_n $. If $n$ is even, then by Lemma~\ref{abelian-cyclic-quotient} we are done by the same argument as for type $B_n $.
Suppose that $n$ is odd. Since $\bG$ is of type $C_n $, we may assume that $F$ acts trivially on $I$.
Then, by Lemma \ref{dynkin-con} (ii), we get that there are at least two $wF$-orbits on $J$.
But $Z\leq Z(\bG)$ is cyclic, and by Lemma \ref{abelian-cyclic-quotient}, for each $wF$-orbit $J'$ of $J$, $[\bL_{J'}, \bL_{J'} ]^F$ intersects $Z$ non-trivially, a contradiction.

Suppose $\bG $ is of type $D_n $. Since $|J_1|=t $ is odd, $J_1 \ne \{ \alpha_{n-1}, \alpha_{n} \} $. Hence, by Lemma~\ref{dynkin-class},
$n $ is even and $J_1$ is one of $\{ \alpha_1, \alpha_3, \ldots, \alpha_{n-3}, \alpha_{n-1}\} $ or $\{ \alpha_1, \alpha_3, \ldots, \alpha_{n-3}, \alpha_{n} \}$. In any case $J=J_1 $ and $n=2t $. If $G $ is of type $\,^2D_n $, we may assume that $F(\alpha_{n-1} ) = \alpha_n $ and $F(\alpha_i)=\alpha_i $ for all $i$, $1\leq i\leq n-2 $. But then $ \,^wF(J) \ne J$ for any $w\in W$, hence there is no $F$-stable Levi subgroup corresponding to $J$. By \cite[Prop.~10]{carter78}, $ Z^{\circ} (\bL)^F \cong Z^{\circ} (\bL_I)^{wF} $ is cyclic of order $ (q^t -1)$ which proves all the assertions for the case $\bG $ is of type $D_n$.

If $\bG $ is of type $E_7$, then by Lemma \ref{dynkin-class}, $J=J_1 =\{ \alpha_1, \alpha_2, \alpha_4\}$.
\end{proof}

%%%%%%%%%%%%%%%%%%%%%%%%%%%%%%%%%%%%%%%%%%%%%%%%%%%%%%%%%%%%%%%%%%%

\section{On Bonnaf\'e-Rouquier equivalences and central extensions}
\label{br}

 We keep the notation of Section 3. We assume throughout this section also that the fixed point subgroup of no $F$-orbit of components of $\bG$ is isomorphic to a Suzuki or Ree group. We assume in this section that $\ell $ is a prime different from $p$. We briefly recall the standard set-up for describing the representation theory of $G$.

Let $\bT$ be an $F$-stable maximal torus of
$\bG$, and $\bG^*$ be a group in duality with $\bG$ with respect to $\bT$ and with corresponding Steinberg homomorphism again denoted by $F$. Set $G^*=\bG^{*F}$. By the fundamental results of Lusztig, the set of complex
irreducible characters of $G$ is a disjoint union of rational Lusztig series ${\mathcal E}(G, s)$, where
$s$ runs over semi-simple elements of $G^*$ up to conjugation. Further, by results of Brou\'e-Michel and Hiss, for any $\ell$-block $B$ of $G$, there is a unique $G^*$-conjugacy class of $\ell' $-elements $s$ such that $B$ contains a complex irreducible character in ${\mathcal E}(G, s)$. We will call such an $s$ a semi-simple label of $B$.

Let $\bL$ be an $F$-stable Levi subgroup of $\bG$ such that $C_{\bG^*} (s) \leq \bL^*$, where $\bL^*$ is a Levi subgroup of $\bG^*$ in duality with $\bL$.
Then Lusztig induction provides a one-one correspondence between $\ell$-blocks of $\bL^F$ with semi-simple label $s$ and $\ell$-blocks of $G$ with semi-simple label $s$.
If a block $B$ of $ G$ and a block $C$ of $\bL^F $ are related in this way, then $B$ and $C$ are said to be Bonnaf\'e-Rouquier correspondents
(see Definition 7.7 of \cite{kessarmalle}).
By \cite[Theorem B', \S 11]{BonRou}, if $B$ and $C$ are Bonnaf\'e-Rouquier correspondents, then $B$ and $C$ are Morita equivalent. We show next that this equivalence is preserved on passing to quotients by central $\ell$-subgroups.

\begin{proposition}\label{bonnrou-cent} Let $\bL$ be an $F$-stable Levi subgroup of $\bG$,  $L :=\bL^F$,  $B$ a block of $\CO G$ and $C$ a  block of $\CO L $ in Bonnaf\'e-Rouquier correspondence with $B$. Let $ Z $ be a central $\ell$-subgroup of $G$ contained in $L$, and let $\bar B$ (respectively $\bar C$) be the block of $\CO (G /Z)$ (respectively $\CO (L/Z)$) corresponding to $B$ (respectively $C$).
Then $\bar B$ and $\bar C$ are Morita equivalent.
\end{proposition}

\begin{proof}
%As described in Section~10 of \cite{BonRou}, the cohomology bimodule inducing the Bonnaf\'e-Rouquier Morita equivalence between $\CO H c $ and $\CO Gb$
%is derived from a Deligne-Lusztig variety consisting of certain cosets in $G$ on which $G$ acts on the left and $H$ on the right.
%For each coset $gT$ of any subgroup $T \leq G$ and any $z \in Z$ we have $zgT=gTz$, hence the same is true for the resulting bimodule.
%The result follows by Lemma \ref{moritaquotient}.
%{\bf Alternate explanation: }
Let $\bf V$ be the unipotent radical of some parabolic subgroup of $\bG$ containing $\bL$ as Levi complement.
By \cite[Theorem B', \S 11]{BonRou}, $C $ and $B$ are isomorphic via a $(B, C) $-bimodule $M$ which is isomorphic to a direct summand of the $\ell$-adic cohomology module
$H_c^r({\bf X}, {\mathbb Z}_\ell) \otimes_{{\mathbb Z}_\ell } \CO $, where $ {\bf X}$ is the $(G, L)$-variety consisting of cosets $ g {\bf V} $ in $\bG$
such that $g^{-1}F(g) \in {\bf V}\cdot \,^F{\bf V} $ and where the action of $G$ is by left multiplication and that of $L$ is by right multiplication.
By Lemma \ref{moritaquotient} and the functoriality of $\ell$-adic cohomology with respect to finite morphisms, it suffices to note that $z g {\bf V} =g{\bf V }z $ for all $ g \in \bG$ and all $ z \in Z$ (here we use the fact that
$Z(G) \leq Z(\bG)$).
\end{proof}

The next result gives a sufficient condition for an $\ell$-block of $G$ to be nilpotent-covered.

\begin{lemma}\label{lem:nilp-cov-char} Let $B$ be a block of $\CO G$ and let $ s \in G^*$
be a semi-simple label of $B$. If $C_{\bG^*}^{\circ} (s)$ is a torus, then there exists a finite group $\tilde G$ such that $ G \unlhd \tilde G$, $Z(G) \leq Z(\tilde G) $ and a nilpotent block $\tilde B$ of $\CO \tilde G$ covering $B$. If $C_{\bG^*}(s)$ is a torus, then $B$ is nilpotent.
\end{lemma}

\begin{proof} Let $\iota: \bG \to \tilde\bG $, where $\tilde \bG$ is a connected reductive group with connected centre such that $[\tilde\bG, \tilde \bG] =[\bG, \bG] $, with corresponding Steinberg homomorphism again denoted $F$
 and let $\iota^*: \tilde \bG^* \to \bG^*$ be the corresponding dual map (see Section 2, \cite{Bon2006}). Set $\tilde G= \tilde \bG^F$. Then $Z(G) = Z(\bG)^F \leq Z(\tilde \bG) ^F$. Let $\tilde B $ be a block of
$\CO \tilde G$ covering $B$ and let $\tilde s\in \tilde \bG^{*F}$ be a semi-simple label of $\tilde B$. We may assume that
 $s= \iota^*(\tilde s)$ and hence that $C_{\tilde \bG^*}(\tilde s)= C_{\tilde \bG^*}^{\circ}(\tilde s) $ is the inverse image of $C_{\bG^*}^{\circ} (s)$ in $\tilde \bG^*$(see \cite[Prop.~11.7]{Bon2006}). So, $C_{\tilde \bG^*} (\tilde s)$ is a maximal torus of $\tilde \bG^*$.
By the Bonnaf\'e-Rouquier theorem, it follows that $\tilde B $ is Morita equivalent to
a block of $C_{\tilde \bG^*} (\tilde s)^F$. Since $C_{\tilde \bG^*} (\tilde s)^F$ is an abelian group and Morita equivalence preserves nilpotency of blocks (see \cite[8.2]{pu99}),
 $\tilde B$ is nilpotent. The second assertion follows by the same argument replacing $\tilde \bG$ by $\bG$.
\end{proof}

%%%%%%%%%%%%%%%%%%%%%%%%%%%%%%%%%%%%%%%%%%%%%%%%%%%%%%%%%%%%%%%%%%%%%

\section{On abelian $2$-blocks of finite reductive groups in odd characteristic}
\label{reductive groups}

We keep the notation and assumptions of the previous section.
Our aim in this section is to determine the structure of all $2$-blocks of $G$ when $p$ is odd and $\bG$ is simply connected
whose defect groups become abelian on passing to the quotient by a central $2$-subgroup of $G$. This will be done in Proposition \ref{2abelian-simp}.

We first consider quasi-isolated blocks. Recall that the block $B$ is called quasi-isolated if $C_{\bG^*}(s)$ is not contained in any proper Levi subgroup of $\bG^*$. In this case we call $s$ a quasi-isolated element.

\begin{lemma}\label{quasi-A-nilp} Suppose that ${\bf G}= PGL_{n+1}({\mathbb F})$ is the adjoint group of type $A_{n}$ and let
 $ s$ be a semi-simple quasi-isolated element of $\bf G$. Then, $o(s) $ is a divisor of $n+1$, $C_{\bf G}^{\circ} (s) $ is a Levi subgroup of ${\bf G}$ and $C_{\bf G}^{\circ} (s) /Z( C_{\bf G}^{\circ} (s)) $ is a direct product of $o(s)$ adjoint groups of type $A_{\frac{n+1} {o(s)} -1}$, that is, a direct
 product of $o(s)$ copies of $ PGL_{\frac{n+1}{o(s)} } ({\mathbb F})$.
\end{lemma}

\begin{proof} See table 2 of \cite{Bon2004}.
\end{proof}

\begin{lemma} \label{quasi-iso-ab} Assume that $p$ is odd and that $\bG$ is simple and simply-connected.
 Let $B$ be a quasi-isolated $2$-block of $G$ with semi-simple label $s \in \bG^{*F}$.
\begin{enumerate}[\rm(a)]
\item Suppose that $B$ has abelian defect groups. Then one of the following holds.
\begin{enumerate}[(i)]
\item $\bG $ is of type $A_n $,
 $n$ is even, and $C_{\bG^*}^{\circ} (s) $ is a torus.
\item $\bG $ is of type $G_2 $, $F_4 $, $E_6$ or $E_8$, $s=1 $ and $B$ is of defect $0$.
\end{enumerate}
\item Suppose that $B$ has non-abelian defect groups, but for some non-trivial central $2$-subgroup $Z$ of $\bG^ F$, the image $\bar B$ of $B$ in $G/Z$
has abelian defect groups. Then $Z $ is cyclic of order $2$ and one of the following holds.
\begin{enumerate}[(i)]
\item $\bG$ is of type $A_n$, $ n\equiv 1\pmod{4}$ and the defect groups of $\bar B $ are $C_2 \times C_2 $.
\item $\bG$ is of type $E_7$, and the defect groups of $\bar B $ are $C_2 \times C_2 $.
\end{enumerate}
\end{enumerate}
\end{lemma}

\begin{proof} Suppose first that $\bG$ is of type $B_n$, ($n\geq 2 $), $C_n$, ($ n\geq 3 $) or $D_n$, $(n \geq 4)$. The identity is the only quasi-isolated element of odd order in $\bG^*$, (see for example table 2 of \cite{Bon2004}), hence $s=1 $ and $B$ is unipotent. Further, the only unipotent $2$-block in classical groups is the principal block \cite{CaEn93} (for the case $G=\,^3D_4(q)$, see the relevant section of \cite{En00}). In particular, the defect groups of $B$ are the Sylow $2$-subgroups of $G$. Since the Sylow $2$-subgroups of $G/Z$ are not abelian for any central subgroup of $G$, neither $B$ nor $\bar B$ has abelian defect groups.

Suppose that $\bG$ is of type $G_2 $, $F_4 $, $E_6$ or $E_8$. Then $G$ has either a trivial centre or a centre of order $3$, so the hypothesis of (b) does not hold. By the lists in \cite{En00} and \cite{kessarmalle}, one sees that the only possibility for (a) to hold is $s=1 $ and $B$ of defect zero.

 Suppose that $\bG$ is of type $E_7$ and $s=1 $. By \cite{En00}, the defect groups of $B$ are non-abelian. If $Z $ is cyclic of order $2$, then $\bar B $ has abelian defect groups only if $\bar
B$ corresponds to lines 3 or 7 of the table on page 354 of \cite{En00} in which case the defect groups of $B$ are dihedral of order $8$ and those of $\bar B$ are isomorphic to $C_2 \times C_2 $. If $\bG $ is of type $E_7 $ and $s \ne 1 $, then by \cite{kessarmalle}, neither $B$ nor $\bar B$ has abelian defect groups.

It remains to consider the case that $\bG$ is of type $A_n$. Then $C_{\bG^*}^{\circ}(s) $ is a Levi subgroup of $\bG^*$. 
Let $\bH$ be an $F$-stable Levi subgroup of $\bG$ in duality with $C_{\bG^*}^{\circ}(s) $. By
\cite[Proposition 1.5]{En08}, the Sylow $2$-subgroups of $\bH^F$ are defect groups of $B$.
Suppose first that the Sylow $2$-subgroups of $\bH^F$ are abelian. Then
by Lemma~\ref{abelian-cyclic-quotient}(i) applied to $\bH$, $\bH$ and hence $C_{\bG^*}^{\circ}(s) $ are tori. Since $s$ has odd order,
Lemma \ref{quasi-A-nilp} implies that $n=o(s)-1$ is even.
Now suppose that the Sylow $2$-subgroups of $\bH^F$ are non-abelian, and let $Z $ be a central $2$-subgroup of $G$ such that $\bH^F /Z$ has abelian Sylow $2$-subgroups
(note that since $\bH$ is a Levi subgroup of $\bG$, $Z \leq \bH$).
By Lemma \ref{abelian-cyclic-quotient}(ii) all components of $\bH$ and hence of $C_{\bG^*}^{\circ}(s)$ are of type $A_1 $. Thus, by Lemma \ref{quasi-A-nilp}, $n \equiv 1\pmod{4}$. Since $\bar B$ has abelian defect groups and $B$ has non-abelian defect groups,
 by Proposition \ref{2abelian-levi} and its proof, the defect groups of $\bar B$ have order $4$. On the other hand, since $\bH^F/Z$ has a subgroup isomorphic to $C_2 \times C_2 $, the defect groups of $\bar B$ are isomorphic to $C_2 \times C_2 $ as required.
\end{proof}

The notation $A_n(q)$, $D_n(q)$ etc.\ in the following proposition stands for the simply connected version of the groups in question.

\begin{proposition}\label{2abelian-simp} Suppose that ${\mathbb F} $ has odd characteristic, $\bG$ is simple and simply connected. Let $B$ be a $2$-block of $\bG$ with semi-simple label $s$ and defect group $P$. Suppose that $Z \leq O_2(Z(G))$ is such that $P/Z$ is abelian and that $C_{\bG^*}^{\circ}(s) $ is not a torus. Then, one of the following holds.
\begin{enumerate}[(i)]
\item $\bG $ is of type $G_2$, $F_4 $, $E_6 $ or $E_8 $, $B$ is unipotent and $P=1 $.
\item $\bG$ is of type $E_7$, $B$ is quasi-isolated, $Z\ne 1 $ and $P/Z $ is a Klein $4$-group.
\item $\bG$ is of type $A_n $, $G \cong A_n(q) $ or $\,^2A_n(q) $, $n+1= 2t$, $t$ odd, $q \equiv \pm 3\pmod{8}$, $P$ is quaternion of order $8$ and $P/Z $ is a Klein $4$-group.
\item $\bG$ is of type $E_7 $ or $E_8 $ and there exists an $F$-stable Levi subgroup $\bL$ of $\bG $, and a $2$-block $C$ of $\bL^F $ such that $B$ and $C$ are Bonnaf\'e-Rouquier correspondents, $[\bL, \bL]^F \cong E_6(q)$ or $\, ^2E_6(q) $ and $Z^{\circ}(\bL)^F$ contains a defect group of $C$.
\item $\bG $ is of type $D_n $, $G \cong D_n(q)$, $ n =2t $, $ t$ odd, and $q \equiv \pm 3\pmod{8}$ or $\bG $ is of type $E_7 $. In these cases, there exists an $F$-stable Levi subgroup $\bL$ of $\bG $, and a $2$-block $C$ of $\bL^F $ such that $B$ and $C$ are Bonnaf\'e-Rouquier correspondents and
the following holds:
$ \bL^F = Z^{\circ} (\bL)^F \times [\bL, \bL]^F $ and $ [\bL, \bL]^F \cong SL_2 (q^t) $.
 Let $C_0$ (respectively $C_1$) be the $2$-block of
$ Z^{\circ} (\bL)^F $ (respectively $ [\bL, \bL]^F $) covered by $C$ and let $P_i $, $i=0,1 $ be a defect group of $C_i $.
Then $C_1 $ is the principal block of $[\bL, \bL]^F $, $P_1 $ is quaternion of order $8$,
$ Z = (Z\cap P_0) \times (Z\cap P_1)$ and $Z\cap P_1 \ne 1 $. In particular, $P_1/ (Z\cap P_1) $ is a Klein $4$-group. Further, if
 $\bG $ is of type $D_n $, then $P_0$ is cyclic of order $2$ or $4$.
\end{enumerate}
\end{proposition}

\begin{proof} Suppose first that $\bG$ is of classical type $A_n $, $B_n$, $C_n $, $D_n$. Since $s$ has odd order and $2$ is the only bad prime for classical groups, $C_{\bG^*}^{\circ} (s)$ is a Levi subgroup of $\bG^*$ which is also $F$-stable since $s \in \bG^{*F}$. Denote by $\bL$ an $F$-stable Levi subgroup of $\bG $ in duality with $C_{\bG^*}^{\circ}(s)$.
By \cite[Prop.~1.5]{En00} we may assume that $P$ is a Sylow $2$-subgroup of $\bL$. So, by Proposition \ref{2abelian-levi}, $\bG $ is not of type $B_n $ or $C_n $. If $\bG$ is of type $A_n$, then by Proposition \ref{2abelian-levi}, case (iii) of the proposition holds.

Suppose that $\bG $ is of type $D_n $. Then $Z(\bG) $ is a $2$-group, hence $C_{\bG^*}(s) /C_{\bG^ *}^{\circ} (s) $ is a $2$-group. On the other hand, the exponent
of $ C_{\bG^*}(s) /C_{\bG^ *}^{\circ} (s) $ divides the order of $s$ which is odd. Hence, $C_{\bG^*}(s) =C_{\bG^*}^{\circ} (s) =\bL^*$ and there is a $2$-block of $\bL^F $, say $C$, which is a Bonnaf\'e-Rouquier correspondent of $B$. Moreover, since $s$ is central in $\bL^*$, there exists a linear representation $\hat s$ of $\bL^F$ with the same order as $s$ and a unipotent block $C'$ of $\bL^F$ such that
the characters of $C$ are of the form $\hat s \otimes \chi $, where $\chi $ is a character of $C'$. Since the principal block
in a group of type $A_n$ is the only unipotent block, it follows from Proposition \ref{2abelian-levi}(b) that case (v) of the proposition holds.

Thus, $\bG$ is of the exceptional type. Let $\bL $ be an $F$-stable Levi subgroup of $\bG$ and $\bL^* $ a Levi subgroup of $\bG^*$ in duality with $\bL$ such that $s$ is a quasi-isolated element of $\bL^* $ and $C_{\bG^*}(s) =C_{\bL^*}(s)$.
Let $C $ be a $2$-block of $\bL^F $ in Bonnaf\'e-Rouquier correspondence with $B$.

Suppose that all components of $\bL $ are of type $A$. Then, by the same argument as in the beginning of the proof, $C_{\bG^*}^{\circ} (s)= C_{\bL^*}^{\circ} (s) $ is a Levi subgroup of $\bL^* $, hence of $\bG^* $ which is also $F$-stable.
Let $\bH \leq \bL \leq \bG $ be an $F$-stable Levi subgroup in duality with $ C_{\bL^*}^{\circ} (s) $. Again, by
\cite[Prop.~1.5]{En00}, a Sylow $2$-subgroup, say $P'$ of $\bH^F $ is a defect group of $C$. By \cite[Theorem~1.3]{kessarmalle}
$P'/Z \cong P/Z$ is abelian. By Proposition \ref{abelian-cyclic-quotient}, $Z\ne 1$, hence $\bG $ is of type $E_7 $. In this case,
 $Z(\bG) $ is of order $2$. So, by the same argument as for the case $D_n$ above, $C_{\bG^*} (s) = C_{\bL^*}^{\circ} (s) =\bL^*$. Now again by the same argument as given above for type $D_n $ and using Proposition \ref{2abelian-levi}(c), we get that case (v) of the proposition holds.

We assume from now on that $\bG$ is of exceptional type and $\bL $ has a component which is not of type $A$.
Let $\Delta $ be the set of components of $[\bL, \bL]$ and let $Orb_F (\Delta)$ be the set of $F$-orbits of $\Delta $.
Then,
\[[\bL, \bL]= \prod_{ \delta \in Orb_F(\Delta)} \prod_{\bX \in \delta} \bX \]
so that
\[ [\bL, \bL] ^ F = \prod_{ \delta \in Orb_F(\Delta)} (\prod_{\bX \in \delta} \bX)^F. \]
Let $E$ be a block of $[\bL, \bL]^F$
covered by $C$ and for each ${\delta} \in Orb_F(\Delta)$, let $E_{\delta}$
be the corresponding component block of $ (\prod_{\bX \in \delta} \bX)^F$ and $D_{\delta} $ a defect group of $E_{\delta}$.

Let $ \iota^* : \bL^* \to [\bL, \bL]^*$ be a map dual to the inclusion $[ \bL, \bL] \to \bL$ chosen to be compatible with $F$. Then
\[[\bL, \bL]^*= \prod_{ \delta \in Orb_F(\Delta)} \prod_{\bX \in \delta} \bX^* \]
where for each $\delta \in Orb_F (\Delta)$, and $\bX \in \delta $, $\bX^*$ is dual to $\bX$.
Let $s\in \bL^{*F}$ be a semi-simple label of $C$ and let $\bar s =\prod_{\delta \in Orb_F(\Delta)} \prod_{\bX \in \delta} \bar s_{\bX}$ be the image of $s$ under $\iota^*$. Then for each $\delta \in Orb_F(\Delta) $, $\prod_{\bX \in \delta} s_{\bX} $ is a semi-simple label of $E_{\delta}$.
Since $s$ is quasi-isolated in $\bG^*$, $\bar s $ is quasi-isolated in $[\bL, \bL]^F$ (see \cite[Prop.~2.3]{Bon2004}), and consequently, $\prod_{\bX \in \Delta} s_{\bX} $ is quasi-isolated in
 $\prod_{\bX \in \delta} \bX^* $.
Further, if $\bX \in \delta$, then through the isomorphism of $ (\prod_{\bY \in \delta} \bY)^F$ with $\bX^{F^{|\delta|}} $ induced by projection onto $\bX$, $E_{\delta}$ is identified with a block of $\bX^{F^{|\delta|}} $ with quasi-isolated semi-simple label $s_{\bX}$.

Let $\delta \in Orb_F(\Delta)$ and $\bX \in \delta $ be such that $\bX$ is not of type $A_n$.
By Lemma \ref{quasi-iso-ab},
$\bX $ is not of type $B_n$, $C_n$ or $D_n$.
If $\bX $ is of type $G_2 $ or $F_4 $, then $ \bL = \bG $, $ Z=1 $ and by Lemma \ref{quasi-iso-ab}, $B$ is unipotent, and $P=1$, so (i) of the proposition holds.

Suppose $\bX $ is of type $E_6 $. Then $\bG $ is of type $E_6$, $ E_7 $ or $E_8$ and by rank considerations, $\delta =\{\bX\} $.
Further, $Z\cap X =1$. So $D_{\delta}$ is abelian and
by Lemma \ref{quasi-iso-ab} , $E_{\delta}$ is a unipotent block of defect $0$.
If $\bG $ is of type $E_6 $, then $ \bG =\bL= \bX $, whence $B$ is a unipotent block with trivial defect groups, that is (i) holds.
Suppose $\bG $ is of type $E_7$. By rank considerations,
$ \bL = Z^{\circ} (\bL) \bX $.
Let $ M= Z^{\circ}(\bL)^F \bX^F $. Then $M$ is a normal subgroup of $ \bL^F $ and $\bL^F/M$ is abelian of order
$| Z^{\circ} (\bL)^F \cap \bX^F| $. In particular $[ \bL^F :M]$ is $1$ or $3$, which means that $D \leq M$. Since $D_{\delta}=1$,
 and since $D$ is a defect group of a block of $M$ covered by $C$ and $Z^{\circ}(\bL)^F$ and $\bX^F$ commute with each other, we may assume that $D \leq Z^{\circ}(\bL) ^F$. So (iv) of the proposition holds.

Now suppose that $\bG $ is of type $E_8 $. Either $ \bL = Z^{\circ} (\bL)\bX$ or $[\bL, \bL] =\bY \times \bX $, where $\bY$ is of type $A_1$. In the former case, we argue as above to conclude that case (iv) holds. Since $Z=1 $, by Lemma \ref{quasi-iso-ab}, $\bL$ has no component of type $A_1$. This rules out the latter case.

 Suppose that $\bX $ is of type $E_7 $. Then $\bG $ is of type $E_7$ or of type $E_8 $ and $\delta $ consists of a single component. Suppose $\bG $ is of type $E_7$. Then $ \bL =\bG$ and hence by Lemma \ref{quasi-iso-ab}, $B$ is quasi-isolated and $P/Z \cong C_2 \times C_2$, that is case (ii) holds. The case that $\bG =E_8 $ is ruled out by Lemma \ref{quasi-iso-ab} since if $\bG$ is of type $E_8$, then $Z=1 $.

Finally, suppose that $\bX $ is of type $E_8$. Then, $\bG =\bL $ is of type $E_8$ and by Lemma~\ref{quasi-iso-ab}, $B$ is unipotent and $P$ is trivial, so case (i) holds.
\end{proof}

{\bf Remark.} Some special cases of the above theorem are given in \cite[Theorem 1.3]{KKL}. We point out that (iii) of \cite[Theorem 1.3]{KKL} is not correct for the classical groups of type $D_n(q)$ and $r=3$: the above analysis shows that if $n=2t$, $q\equiv -3\pmod{8}$, then there do exist non-nilpotent $2$-blocks with elementary abelian defect groups of order $8$ in the simple group of type $D_n(q)$. This does not pose a problem for the conclusion of Theorem 1.1 of \cite{KKL} as by part (ii) of the above proposition such a block has $3$ simple modules, and the conclusion of Theorem 1.1 of \cite{KKL} for groups of type $D_n(q)$ follows by Theorem 5.1 and Proposition 3.2 of \cite{KKL}.

We record the following for later use.

\begin{proposition}\label{nilp-cov-nilp} Suppose that ${\mathbb F} $ has odd characteristic, $\bG$ is simple and simply connected. Let $B$ be a $2$-block of $G$ with semi-simple label $s$ and defect group $P$. Suppose that $Z \leq O_2(Z(G))$ is such that $P/Z$ is abelian and that $C_{\bG^*}^{\circ} (s) $ is a torus, but $C_{\bG^*}(s) $ is not connected. Then, $\bG$ is of type $A_n $ or $ E_6 $. Further, if $\bG $ is of type $E_6 $, then
there exists an $F$-stable Levi subgroup $\bL$ of $\bG $ and a $2$-block $C$ of $\bL^F$ in Bonnaf\'e-Rouquier correspondence with $B$ such that $ [\bL, \bL]^ F $ is a direct product of at most two groups of type $A_2$. In particular, if $\bG $ is of type $E_6$ (so that $Z=1$) and
$P$ is elementary abelian of order $16$ then a defect group $P'$ of $C$ has the form $P'= P_1 \times P_2 $, $P_1\cong P_2 \cong C_2 \times C_2 $ and both $P_1$ and $P_2 $ are invariant under $N_{\bL^F}(P')$.
\end{proposition}

\begin{proof} We freely use the notation of the proof of Proposition \ref{2abelian-simp}.
Let $\bL$ be a Levi subgroup of $\bG$ in duality with an $F$-stable Levi subgroup of $\bG^* $ containing $C_{\bG}^*(s) $ and such that $s$ is quasi-isolated in $\bL$.
The exponent of $C_{\bG^*} (s) / C_{\bG^*}^{\circ} (s)= C_{\bL^*} (s) / C_{\bL^*}^{\circ} (s) $ is a divisor of the order of $s$ (see \cite[13.15]{DiMi})
and also $C_{\bL^*}(s)/ C_{\bL^*}^{\circ} (s) $ is isomorphic to a subgroup of $ Z(\bL )/Z^{\circ}(\bL) \leq Z(\bG)/Z^{\circ} (\bG)$. Since $s$ has odd order and $Z(\bG) $ is a $2$-group unless $\bG $ is of type $A_n $ or $E_6 $, the first assertion follows.

Now, suppose that $\bG$ is of type $E_6 $, so $C_{\bL^*}(s)/ C_{\bL^*}^{\circ} (s) =Z(\bL )/Z^{\circ}(\bL) $ has order $3$ and $3$ divides the order of $s$. By \cite[Table 3]{Bon2004}, $\bG^* \ne \bL^* $. So, $\bL $ is a proper Levi subgroup of $\bG$ and in particular has semi-simple rank at most $5$. Let $X$ be a component of $[\bL, \bL]$. Then,
$s_{\bX}$ is a quasi-isolated element of $\bX^*$, so as explained in the proof of Proposition \ref{2abelian-simp},
$\bX $ is not of classical type $D_n$. Thus, $\bX $ is of type $A_n$. By Lemma \ref{quasi-A-nilp},
 $n+1 = o(s_{\bX}) $. Since $1 \leq n \leq 5 $ and $o(s_{\bX}) $ is odd and a multiple of $3$, it follows that
 $n=2 $.

Hence all components of $[\bL, \bL]$ are of type $A_2 $ and there are at most two components. Since the $2$-rank of special linear or unitary groups in odd dimension is $2$, the final assertion is immediate from the structure of $\bL^F$.
\end{proof}

\section{Structure Theorem}
\label{main theorem}

Our aim in this section is to prove the following.
\begin{theorem} \label{ab-2-str}  Let $\ell=2 $ and let $G$ be a quasi-simple group. If $B$ is a block of $\CO G$ with abelian defect group $P$, then one (or more) of the following holds:
\begin{enumerate}[(i)]
\item $G/Z(G)$ is one of $A_1(2^a)$, $\,^2G_2(q)$ (where $q \geq 27$ is a power of $3$ with odd exponent), or $J_1$, $B$ is the principal block and $P$ is elementary abelian.
\item $G$ is $Co_3 $, $B$ is a non-principal block, $P \cong C_2 \times C_2 \times C_2 $ (there is one such block).
\item There exists a finite group $\tilde G$ such that $G \unlhd \tilde G $, $Z(G) \leq Z(\tilde G) $ and such that $B$ is covered by a nilpotent block of $\tilde G$.
\item $B$ is Morita equivalent to a block $C$ of   $\CO L $   where $L=L_0 \times L_1$  is  a subgroup of $G$ such that the following holds: The defect groups of $C$ are isomorphic to $P$, $L_0$ is abelian and the block of $\CO L_1 $ covered by $C$ has Klein $4$-defect groups.
In particular, $B$ is Morita equivalent to a tensor product of a nilpotent block and a block with Klein $4$-defect groups.
\end{enumerate}
\end{theorem}

\begin{proof} If $P$ is central in $G$, then $B$ is nilpotent and we are in case (iii). Hence we may assume that $P$ is non-central. We consider first the case $|P| \leq 8$ or $P$ is cyclic. If $P$ is cyclic or if $P \cong C_4 \times C_2$, then $B$ is again nilpotent, hence we are in the situation of (iii).
If $P$ is a Klein $4$-group, then we are in case (iv) (with $L_1=L=G$). Thus, we may assume that $P$ is a non-cyclic group of order at least $8$ and that if $|P|=8 $, then $P$ is elementary abelian.

%If $B$ is the principal block of $G$, then by Walter's theorem~\cite{wal69} $G/Z(G)$ is one of $A_1(2^a)$, %$\,^2G_2(q)$, $J_1$ or $PSL_2(q)$ for $q \equiv 3$ or $5\pmod{8}$. If $G/Z(G)=PSL_2(q)$, then the Sylow %$2$-subgroups of $ G$ For the other three type of groups, we note that $Z(G)$ is a group of odd order, hence %in each case
%In each case $P/Z$ is elementary abelian, and in the last case $P/Z \cong C_2 \times C_2$. Note that in each case %the Schur multiplier of the simple group has odd order, so $Z=1$. Hence we are in case (i) or (iv) of the %statement. Hence we may assume that $B$ is a non-principal block.

Write $\bar G:= G/Z(G)$. Suppose $\bar G$ is a sporadic group. If $B$ is the principal block, then $\bar G =G =J_1$ and $P$ is elementary abelian of order $8$, hence we are in case (i). The non-principal $2$-blocks of quasi-simple groups with sporadic quotient are listed in~\cite{No}, and the only possibility with $|P|\geq 8 $ is $\bar G =G = Co_3 $ and $P$ elementary abelian of order $8$. Thus, $B$ is as in (ii).

Suppose $\bar G:= G/Z(G)$ is an alternating group $A_n $, $n \geq 5 $. The trivial group and $C_2 \times C_2 $ are the only abelian $2$-groups occurring as defect groups of $2$-blocks of $\bar G$, hence $G \ne \bar G$. If $G$ is a double cover of $\bar G$, then the lift of a Klein $4$-block of $\bar G$ has non-abelian defect groups of order $8$ and the lift of a defect zero block has central defect groups. Suppose that $G$ is an exceptional cover of $A_n$.
If $\bar G= A_6, A_7 $ and $G$ is a $3$- or $6$-fold covering of $\bar G$, then the Sylow $2$-subgroups of $G/Z(G)$ are non-abelian of order $8$ and consequently $G$ has no blocks with $P$ as defect group for which we have not accounted.

Suppose $\bar G$ is a finite group of Lie type in characteristic $2$ not isomorphic to any of $\,^2F_4(2)'$, $B_2(2)' $ or $G_2(2)' $ or $PSp_4(2)'$.
Then by \cite[Proposition 8.7]{CEKL}, the non-trivial defect groups of $2$-blocks are Sylow $2$-subgroups of $G$.
The only possibility is $\bar G= G = SL_2(2^a) $ and $B$ the principal block and we are in case (i).

If $\bar G$ is the Tits group $\,^2F_4(2)'$, then $\bar G=G$ and by~\cite{gap}, $G$ has precisely three blocks, two of defect $0$ and the principal block.
If $\bar G$ is isomorphic to $PSp_4(2)'$, then $\bar G \cong A_6$, a case we have already considered.
If $\bar G \cong G_2(2)'$, then $\bar G \cong \,^2A_2(3)$, hence $\bar G = G \cong \, ^2A_2(3)$, a case which will be handled below.
If $\bar G $ is isomorphic to $PSL_2(4)$, then $\bar G \cong A_5 $ a case that has been handled above.

Suppose that $\bar G$ is a finite group of Lie type in odd characteristic. By~\cite{gap} no faithful $2$-block of an exceptional covering of $\bar G$ has non-central defect groups which are abelian. Hence we may assume that $G$ is a non-exceptional covering of $\bar G$. So we have $G =\bG^F/Z$, where $ \bG $ is a simple and simply-connected group defined over an algebraic closure of the field of $p$ elements for $p$ an odd prime, $F: \bG \to \bG $ is a Steinberg endomorphism, and $Z \leq Z(\bG^F)$. By replacing $G$ by a suitable central extension and $B$ by a suitable (and Morita equivalent) lift with isomorphic defect groups if necessary, we may assume that $Z$ is a $2$-group. Let $\hat B$ be a $2$-block of $ \bG^F $ lifting $B$. Then $\hat B$ has a defect group $\hat P$ such that $Z \leq \hat P $ and $\hat P/Z =P $.

If $\bar G =\,^2G_2(q) $, $q$ a power of $3$, then, $\bar G = G$ and the Sylow $2$-subgroups of $G$ are elementary abelian of order $8$. Further, the principal block is the unique $2$-block of $G$ with maximal defect groups (see for instance \cite[Proposition 15.2] {KKL}), and so we are in case (i).

Hence we may assume that $\bG^F$ is not $\,^2G_2(q)$. Suppose first that a semi-simple label of $\hat B $ has a connected centraliser which is a torus. Then
by Lemma \ref{lem:nilp-cov-char}, there exists a finite group $ H $ and a nilpotent block $\hat E$ of $H$ covering $\hat B $ such that $ Z \leq Z(\bG^F) \leq Z( H ) $. Let $ \tilde B $ be the block of $\tilde G:=H/Z$ corresponding to $\hat E$. Then $\tilde B $ covers $B$ and $\tilde B $ is nilpotent. Since $\bG^F $ is quasi-simple, $Z(G) = Z(\bG^F)/Z \leq Z(H)/Z \leq Z(\tilde G)$. Hence we are in case (iii).

So, we may assume that the connected centraliser of a semi-simple label of $\hat B$ is not a torus. We apply Lemma \ref{2abelian-simp}
to $\hat B$. If $\hat B$ is as in case (ii) or (iii) of Lemma \ref{2abelian-simp}, then case (iv) of the theorem holds.
If $\hat B$ is as in case (i) or (iv) of Lemma \ref{2abelian-simp}, then $ \hat B$ is nilpotent, hence so is $B$ by~\cite[8.2]{pu99} and we are in case (iii) of the theorem.
Suppose $\hat B$ is as in case (v) of Lemma \ref{2abelian-simp}. Let $\bL$ and $\hat C$ be as in Lemma \ref{2abelian-simp} (replacing $C$ by $\hat C$) and let $C$ be the block of $L:=\bL^F/Z $ corresponding to $\hat C$. By Lemma \ref{2abelian-simp}, $L$ is a direct product of $L_0:= Z^{\circ} (\bL)/ (Z \cap P_0) $ and $ L_1:= [\bL, \bL]^F/(Z \cap P_1)$ and the $2$-block of $ L_1 $ covered by $C$ has Klein $4$-defect groups. By Lemma \ref{moritaquotient}, $B$ and $C$ are Morita equivalent. The defect groups of $B$ and $C$ are isomorphic by \cite[Theorem~ 1.3]{kessarmalle}. So, we are in case (iv) of the theorem.
\end{proof}

{\bf Remark.} Whenever case (iv) of the above theorem holds, then setting $W= O_2(Z(G)) $, we have $W\leq \bL^F/Z$ and by the proof of Lemma \ref{moritaquotient},
 the Morita equivalence between $B$ and $C$ may be realised by a bimodule on which $\Delta (W) $ acts trivially.

\medskip

The principal $2$-blocks of the $\,^2G_2(q)$ all have elementary abelian defect groups of order $8$ and by~\cite{lm80} they all have the same decomposition matrices. It seems to be folklore that they are all Morita equivalent, but we are unable to find a reference for this. We can however easily show that there are only finitely many Morita equivalence classes amongst them:

\begin{lemma}
\label{Donovan_for_Ree}
Consider the groups $\,^2G_2(3^{2m+1})$ for $m \geq 1$. There are only finitely many Morita equivalence classes of blocks amongst the principal $2$-blocks of these groups.
\end{lemma}

\begin{proof}
By~\cite{lm80} the principal $2$-blocks all have the same decomposition matrices. Write $B_m$ for the principal $2$-block of $\,^2G_2(3^{2m+1})$ and let $fB_mf$ be an associated basic algebra. By the proof of~\cite[1.4]{ke04}, there is an $\mathbb{F}_2$-algebra $A_m$ such that $fB_mf \cong k \otimes_{\mathbb{F}_2} A_m$. Now $A_m$ has dimension equal to the sum of the entries of the Cartan matrix, i.e., $76$. Hence $|A_m|=2^{76}$, and so there are only finitely many possibilities for $A_m$, and hence for the Morita equivalence class of $B_m$.
\end{proof}

\section{Proof of Theorem \ref{homocyclictheorem} }
\label{homocyclic}

Let us keep  the  notation   of Theorem \ref{homocyclictheorem}. In particular, we suppose that
$D$ is an  abelian  $2$-group of rank $2$, so that $D$ is isomorphic to a direct product $C_{2^m} \times C_{2^n}$ of two cyclic subgroups $C_{2^m}$ and $C_{2^n}$. Write $b$ for the unique block of $\CO N_G(D)$ with Brauer correspondent $B$. The following facts are known:

\begin{itemize}
\item If $m \ne n$, then Aut$(D)$ is a $2$-group. Thus $B$ is
a nilpotent block. By the main result of~\cite{puig88}, $B$ is then Morita
equivalent to the group algebra $\CO D$.

\item If $m=n=1$, then $D$ is a Klein $4$-group. This case was completed by Erdmann in~\cite{er82} for blocks defined over $k$, and extended to blocks defined over $\CO$ by Linckelmann in~\cite{lin94}, where it is proved that $B$ is Morita equivalent to the group
algebra $\CO D$, to the group algebra $\CO A_4$ or to the
principal block of the group algebra $\CO A_5$. (Here, $A_n$ denotes the
alternating group of degree $n$.) Therefore, in the following we will
assume that $m=n>1$.

%\item If $G$ is solvable, then $B$ is Morita equivalent to $b$ in $N_G(D)$, and $b$ is Morita equivalent to
%the group algebra $FL$ where either $L \cong D$
%or $L \cong D \rtimes E$ with $|E| = 3$.

\item If $D \in {\rm Syl}_2(G)$ (and $m=n>1$), then $G$ is
solvable by a theorem of Brauer (Theorem 1 of~\cite{bra64}).

\item In general, $B$ is perfectly isometric to $b$. This is stated without explicit proof in Remark 1.6 of~\cite{puigusami95}, and a proof is given in Satz 3.3 of~\cite{sambalethesis}. We note that this result does not use the classification of finite simple groups. Consequently, if $m=n>1$, then $B$ is either nilpotent (in which case $l(B)=1$ and $k(B)=|D|$) or $l(B)=3$ and $k(B)=(|D|+8)/3$.

\end{itemize}

The following are well-known, but we state them here for convenience:

\begin{lemma}
\label{normal}
Let $D=C_{2^m} \times C_{2^m}$. Then $\Aut(D)$ is a $\{2,3\}$-group, where $\lvert\Aut(D)\rvert_3=3$. If $\theta \in \Aut(D)$ has order three, then $\theta$ transitively permutes the three involutions in $D$.

Let $G$ be a finite group such that $D \lhd G$. Let $B$ be a block of $\CO G$ with defect group $D$ and let $b_D$ be a block of $\CO C_G(D)$ with $(b_D)^G=B$. Write $N_G(D,b_D)$ for the stabiliser of $b_D$ in $N_G(D)$ and $E_B=N_G(D,b_D)/C_G(D)$. Then $B$ is Morita equivalent to the group algebra $\CO (D \rtimes E_B)$.

If $O_2(Z(G)) \neq 1$, then $B$ is nilpotent.
\end{lemma}

We now give  the proof of Theorem \ref{homocyclictheorem}.

%\begin{theorem}
%Let $G$ be a finite group and let $B$ be a block of $\CO G$ with defect group $D$. Suppose that $D %\cong C_{2^m} \times C_{2^m}$ for some $m \geq 2$. %Then $B$ is Morita equivalent to either %$\CO D$ or $\CO (D \rtimes C_3)$.
%\end{theorem}

\begin{proof}
Let $B$ be a counterexample with $(|G:Z(G)|, |G|)$ minimised in the lexicographic ordering. By the first Fong reduction and minimality, $B$ is quasi-primitive, that is, for any normal subgroup $N$ of $G$, there is a unique block of $N$ covered by $B$. By the second Fong reduction and minimality, $O_{2'}(G)$ is cyclic and central in $G$.

Now suppose that $N:=O_2(G) \neq 1$, so that $N \subseteq D$. Then $B$ covers a unique block $B_C$ of $C:=C_G(N)$, and $B_C$ has defect group $D$. Since $1 \neq N \leq D \cap Z(C)$, the block $B_C$ is nilpotent. Thus by Proposition \ref{kp} and minimality, $C$ is nilpotent. Since $G/C$ is isomorphic to a subgroup of $\Aut(N)$, which is a $\{2,3\}$-group, $G$ is solvable. But we have $O_{2'}(G) \leq Z(G)$, so $C_G(N) \leq NZ(G)$, and $D = N$ since $D$ is abelian. Hence by Lemma \ref{normal} we have a contradiction to minimality.

Hence $O_2(G)=1$, so that $F(G)=O_{2'}(G)=Z(G)=:Z$. Let $N \lhd G$ such that $N/Z$ is a minimal normal subgroup of $G/Z$, and let $B_N$ be the unique block of $N$ covered by $B$. If $B_N$ is nilpotent, then we can use Proposition \ref{kp} again to obtain a contradiction to minimality. Thus we may assume that $B_N$ is not nilpotent, and in particular the defect group $D \cap N$ of $B_N$ is nontrivial. Moreover, we have $D \cap N \cong C_{2^t} \times C_{2^t}$ for some $t \leq m$. This implies that $N/Z$ is the only minimal normal subgroup of $G/Z$, and so $N=F^*(G)$ (the generalised Fitting subgroup).

Assume next that $G$ has a normal subgroup $K$ of index $2$. Let $B_K$ be the unique block of $K$ covered by $B$. Then $B_K$ is $G$-stable, and $B$ is the only block of $G$ covering $B_K$. Moreover, $D \cap K$ is a defect group of $B_K$ and $DK/K \in \Syl_2(G/K)$, so $G=DK$. Hence $2=|G/K|=|DK/K|=|D/D \cap K|$. This implies that $D \cap K$ is a direct product of two non-isomorphic cyclic groups. Hence $\Aut(D \cap K)$ is a $2$-group, and so $B_K$ is nilpotent. By Proposition \ref{kp} and minimality, $K$ is nilpotent. Then $G$ is solvable, a contradiction since $O_2(G)=1$ and $O_{2'}(G) \leq Z(G)$.

Hence $G=O^2(G)$. Let $L_1,\ldots,L_t$ denote the components of $G$. We have seen that these are permuted transitively by $G$, and $L:=L_1 * \cdots * L_t \lhd G$. Let $B_L$ be the unique block of $L$ covered by $B$, and let $B_i$ be the unique block of $L_i$ covered by $B_L$ ($i=1,\ldots,t$). Then $B_L$ has defect group $D \cap L$, and $B_i$ has defect group $D \cap L_i$ ($i=1,\ldots,t$). Thus $D \cap L = (D \cap L_1) \times \cdots \times (D \cap L_t)$, where $D \cap L_1, \ldots,D \cap L_t$ are conjugate in $G$ (since $B_1,\ldots,B_t$ are). This implies that $t \leq 2$. Since $G=O^2(G)$, we must have $t=1$ (by consideration of the kernel of the permutation action). This shows that $G$ has a unique component $L$, so that the layer $E(G)=L$ is quasi-simple and $F^*(G)=L * Z$.

Suppose that $G \neq L$. By the Schreier Conjecture $G/L$ is solvable. Since $G=O^2(G)$, it follows that there is a normal subgroup $N$ of $G$ such that $|G:N|=w$ for some odd prime $w$. Let $B_N$ be the unique block of $N$ covered by $B$. Now $B_N$ is $G$-stable and has defect group $D$. Suppose that $B$ is the unique block of $G$ covering $B_N$. Now $B_N$ has either $1$ or $3$ irreducible Brauer characters, according to whether $B_N$ has inertial index $1$ or $3$ respectively. If the inertial index is $1$, then $B_N$ is nilpotent, and Proposition \ref{kp} and minimality may be applied to obtain a contradiction. Hence $l(B_N)=3$. Similarly $l(B)=3$, since if $l(B)=1$, then $B$ is nilpotent. If $w > 3$, then each irreducible Brauer character in $B_N$ must be $G$-stable, and applying Clifford theory (noting that each simple module extends to $G$ since $G/N$ is cyclic and of odd order) we obtain $l(B)=3w>3$, a contradiction. If $w=3$, then the irreducible Brauer characters in $B_N$ are either all $G$-stable or 
are permuted transitively. If they are all $G$-stable, then as above we have $l(B)=3w=9$, a contradiction. Hence the three irreducible Brauer characters in $B_N$ are permuted transitively and by Clifford theory induce to an irreducible Brauer character, which must lie in $B$, and further we must have $l(B)=1$, a contradiction. Hence $B$ is not the unique block of $G$ covering $B_N$. In this case, since $w$ is an odd prime, every irreducible Brauer character in $B_N$ is $G$-stable, and it follows by Clifford theory (again using the fact that $G/N$ is a cyclic $2'$-group) that $B_N$ is covered by $w$ blocks of $G$ and that each of the three irreducible Brauer characters in $B$ is an extension of a distinct irreducible Brauer character in $B_N$. Hence we have a bijection given by restriction between the irreducible Brauer characters of $B$ and those of $B_N$, and by~\cite[4.1]{HidaKoshitani} $B$ and $B_N$ are Morita equivalent. This gives a contradiction to minimality.

A final application of the second Fong correspondence and minimality show that $Z(G) \leq [G,G]$. Hence we have shown that $G=L$, i.e., $G$ is quasi-simple, and that $Z(G)$ is cyclic of odd order. Proposition \ref{ab-2-str} and Proposition \ref{puig} give an immediate contradiction.
\end{proof}

\begin{corollary}
Let $\ell=2 $, $G$ a finite group  and let  $B$ be a block of  $\CO G$ with abelian defect group $D$ of rank $2$.  If  $D$ is homocyclic of order at least $16$, then there are precisely two Morita equivalence classes of blocks with defect group $D$. If $D$ is a Klein $4$-group, then there are precisely three Morita equivalence classes, and if $D$ is not homocyclic, then $B$ must be nilpotent.
\end{corollary}

\begin{corollary}
Let $D$ be an abelian $2$-group of rank $2$. Then Donovan's conjecture holds for $D$.
\end{corollary}

\section{Donovan's conjecture for $2$-blocks with elementary abelian defect groups}
\label{elabsection}

In this section,   by an $\ell $-block of a finite group $G$, we will mean  a  block of $kG$.
The following is immediate from~\cite[1.11]{du04} and Proposition \ref{kp}:

\begin{proposition}
\label{duvel}
Let $P$ be an abelian $\ell$-group for a prime $\ell$. In order to verify Donovan's conjecture for $P$, it suffices to verify that there are only a finite number of Morita equivalence classes of
quasi-primitive blocks $B$ with defect group $D \cong P$ of finite groups $G$ satisfying the following conditions:
\begin{enumerate}[(i)]
\item $F(G)=Z(G)$,
\item $O_{\ell'}(G) \leq [G,G]$,
\item $G= \langle D^g:g \in G \rangle$,
\item every component of $G$ is normal in $G$,
\item if $N \leq G$ is a component, then $N \cap D \neq Z(N) \cap D$,
\item if $N \lhd G$ and $B$ covers a nilpotent block of $N$, then $N \leq Z(G)$.
\end{enumerate}
\end{proposition}

Applying Proposition \ref{KK} we can reduce further to:

\begin{corollary}
\label{elabreduction}
Let $P$ be an  elementary abelian $\ell$-group for a prime $\ell$. In order to verify Donovan's conjecture for $P$, it suffices to verify that there are only a finite number of Morita equivalence classes of blocks $B$ with defect group $D \cong P$ of finite groups $G$ satisfying the following conditions:
\begin{enumerate}[(i)]
\item $G$ is a central product $G_1 * \cdots * G_t$ of quasi-simple groups;
\item the block $b_i$ of $G_i$ covered by $B$ is not nilpotent.
\end{enumerate}
\end{corollary}

\begin{proof}
It suffices to consider groups $G$ of the form given in Proposition \ref{duvel}. If $G$ has this form, then by the Schreier conjecture there is $N \lhd G$ with $N$ a central product of quasi-simple groups and $G/N$ solvable. By condition (iii) in Proposition \ref{duvel} we have $O^{\ell'}(G)=G$. By Proposition \ref{KK} we may assume that $O^{\ell}(G)=G$. Hence, since $G/N$ is solvable, we may assume that $G=N$ and the result follows from the conditions listed in Proposition \ref{duvel}.
\end{proof}

\begin{theorem}
\label{elabDonovan}
Donovan's conjecture holds for elementary abelian $2$-groups.
\end{theorem}

\begin{proof}
Let $P$ be an elementary abelian $2$-group.

We may assume initially that we have a block $B$ of a group $G$ as in Corollary \ref{elabreduction}, so that $G$ is a central product $G_1 * \cdots * G_t$ of quasi-simple groups. By taking an appropriate central extension of $G$ by a group of odd order, we may assume that $G=E/Z$, where $E \cong G_1 \times \cdots \times G_t$ and $Z \leq Z(G)$ is a $2$-group. Write $Z_i=O_2(Z(G_i))$.

Let $B_i$ be the (unique) block of $G_i$ covered by $B$. Note that $B_i$ has elementary abelian defect groups, and no $B_i$ is nilpotent. Let $D$ be a defect group for $B$ (with $D \cong P$). Then $D_i:=D \cap G_i$ is a defect group for $B_i$. Since $B_i$ is not nilpotent, $|D_i/Z_i| >2$,

Let $B_E$ be the unique block of $E$ corresponding to $B$. Then $B_E \cong B_1 \otimes \cdots \otimes B_t$ and $B_E$ has defect group $D_E \cong D_1 \times \cdots \times D_t$ (in particular $D_E$ is elementary abelian), with $D_E/Z \cong D$.

Then, $B_i$ and $G_i$ belong to one (or more) of the classes (i)--(iv) in Theorem \ref{ab-2-str}. We will define a new group $H$ containing a copy of $Z$ and a block $C$ of $H$ such that $C$ is Morita equivalent to $B_E$ via a bimodule satisfying the conditions in Lemma \ref{moritaquotient}.

Suppose first that $G_i$ satisfies (i), (ii) or (iii) of Theorem \ref{ab-2-str}. Then write $H_i=G_i$ and $C_i=B_i$.

Suppose $B_i$ and $G_i$ are as in (iv) of Theorem \ref{ab-2-str}. Then there is a Morita equivalence with a block $C_i$ of a finite group $H_i$ containing $Z_i$, such that $H_i \cong A_i \times L_i$ where $A_i$ is abelian and $C_i$ covers a block of $L_i$ with Klein $4$-defect groups. Further this Morita equivalence is realised by a bimodule $M_i$ such that $z_im_i=m_iz_i$ for all $z_i \in Z_i$ and all $m_i \in M_i$.

in case (iv) of Theorem \ref{ab-2-str}.

We now have a Morita equivalence satisfying the conditions of Lemma \ref{moritaquotient} from $B_E$ to a block $C$ of the direct product $H=H_1 \times \cdots \times H_t$, where $C$ covers the block $C_i$ of $H_i$, $Z \leq Z(H)$, and $C_i$ is as in (i)--(iv). By Lemma \ref{moritaquotient} this gives a Morita equivalence between $B$ and the unique block $C_{H/Z}$ of $H/Z$ corresponding to $C$. Thus it suffices to assume that $H=E$ and $B_E=C$.

Relabelling as necessary to account for the direct product in case (iv), and noting that blocks of abelian groups are necessarily nilpotent, we may write $H$ as a direct product of groups $H_i$, with block $C_i$ of $H_i$ covered by $C$ satisfying one or more of the following:
\begin{enumerate}[(a)]
\item $C_i$ is a nilpotent-covered block;
\item $C_i$ has defect groups $C_2 \times C_2$ and $O_2(Z(H_i))=1$;
\item $H_i$ and $C_i$ are as in (i) or (ii) of Theorem \ref{ab-2-str}.
\end{enumerate}

Note that we may assume $O_2(Z(H_i))=1$ in case (b), since otherwise $C_i$ is nilpotent and so belongs to case (a). By checking in~\cite{atlas} we see that in case (c) $Z(H_i)$ has odd order.

It follows that $Z$ is contained in the direct product of factors of type (a), i.e., we may express $G$ as a direct product $(U/Z) \times V \times W$, where $U$ is a direct product of groups satisfying (a), $V$ is a direct product of groups satisfying (b), and $W$ is a direct product of groups satisfying (c). Further $B \cong C_{U/Z} \otimes C_V \otimes C_W$, where $C_U$, $C_V$, $C_W$ are the blocks of $U$, $V$, $W$ resp.\ covered by $C$, and $C_{U/Z}$ is the unique block of $U/Z$ corresponding to $C_U$.

Now a tensor product of nilpotent-covered blocks is nilpotent-covered, and by Lemma \ref{nilcoveredandquotients} a quotient of such a block by a central $2$-subgroup is also nilpotent-covered.
Hence $C_{U/Z}$ is nilpotent-covered. So by Proposition \ref{puig}, $C_{U/Z}$ is Morita equivalent to a block with normal elementary abelian defect group, of which there are only finitely many possibilities for Morita equivalence classes.

$C_V$ is a tensor product of a bounded number of blocks with Klein $4$-defect groups, and so there are only a finite number of possibilities for the Morita equivalence class of $C_V$.

By Lemma \ref{Donovan_for_Ree} there are only a finite number of Morita equivalence classes of blocks of groups satisfying (i) and (ii) of Theorem \ref{ab-2-str} with elementary abelian defect groups of order at most $|P|$, and of course the number of factors in $W$ is bounded in terms of $|P|$, hence there are only finitely many possibilities for the Morita equivalence class of $C_W$.

Since $B \cong C_{U/Z} \otimes C_V \otimes C_W$, the result follows.
\end{proof}

\section{On the weak Donovan conjecture for abelian $2$-blocks}
\label{weakdonovansection}

A weak version of  Donovan's conjecture is the following.

\begin{conjecture}  Let $D$ be a finite $\ell$-group.  There is a bound   on the Cartan invariants of blocks of finite groups with defect group $D$   which depends only on $D$.
\end{conjecture}

\begin{theorem} \label {weak-donovan} The weak Donovan conjecture holds for $2$-blocks with abelian defect groups.
\end{theorem}

\begin{proof}  By \cite[Theorem 3.2]{du04}, it suffices to prove that  the weak Donovan conjecture holds for all abelian  defect $2$-blocks  of quasi-simple groups  (note that  the reductions employed in    \cite{du04}  all work in the realm of  abelian defect groups).   By \cite[Theorem 3.9]{lm80} the Cartan  invariants  of  the principal $2$-blocks  of  the Ree groups  are at most  $8$.  The  proof follows by
Theorem \ref{ab-2-str}.
\end{proof}

%%%%%%%%%%%%%%%%%%%%%%%%%%%%%%%%%%%%%%%%%%%%%%%%%%%%%%%%%%%%%%%%%%%%%%%%%%%%%%%%%%%%%%

\section{On numerical invariants for $2$-blocks with elementary abelian defect groups of order $16$}
\label{elabinvariants}

It is shown in~\cite{ks13} that if a block $B$ has elementary abelian defect groups of order $16$, then $k(B)$ is either $8$ or $16$, and that in all but one case, $k(B)$, $k_0(B)$ and $l(B)$ are determined given the action of the inertial quotient on $D$ and certain cocycles. It remains to show that if the inertial quotient has order $15$, then $k(B)=16$ (and consequently $l(B)=15$, $k_0(B)=16$).

\begin{lemma}
\label{inertialquotients}
Let $B$ be a block of a finite group $G$ with elementary abelian defect group $D$ of order $16$ and inertial index $15$. Let $N \lhd G$ have odd prime index, and let $b$ be a $G$-stable block of $N$ covered by $B$. Let $B_D$ be a block of $C_G(D)$ with $(B_D)^G=B$ and $b_D$ be a block of $C_N(D)$ with $(b_D)^G=B$.
\begin{enumerate}[(i)]
\item If $C_G(D) \leq N$, then $N_N(D,b_D)/C_N(D) \leq N_G(D,B_D)/C_G(D)$ and $|N_G(D,B_D):N_N(D,b_D)|$ divides $|G:N|$.
\item If $C_G(D) \not\leq N$, then $N_N(D,b_D)/C_N(D) \cong N_G(D,B_D)/C_G(D)$.
\end{enumerate}
Either $B$ is the unique block of $G$ covering $b$, or there are $|G:N|$ blocks covering $b$.
\end{lemma}

\begin{proof}
Part (i) is immediate.

(ii) Recall that the blocks of $C_N(D)$ with defect group $D$ correspond to inflations of irreducible characters in blocks of defect zero of $C_N(D)/D$ (and similarly for $C_G(D)$).
Let $\theta$ be the canonical character of $b_D$.
%Let $\theta \in \Irr(b)$ be the character corresponding to $b$.
Suppose that $C_G(D) \not\leq N$. So $|C_G(D):C_N(D)|\ne 1$.

Consider first the case $C_G(D) \leq N_G(D,b_D)$. Then $\theta$ extends to $|G:N|$ irreducible characters of $C_G(D)$, each inflated from a block of defect zero of $C_G(D)/D$.
Hence $b_D$ is covered by $|G:N|$ blocks of $C_G(D)$. Let $\theta_1$ be the irreducible character of $B_D$ extending $\theta$. We have $N_G(D,B_D) \leq N_G(D,b_D)$, so $|N_N(D,b_D):C_N(D)| \geq |N_G(D,B_D):C_G(D)|$. Since $C_{15}$ is a maximal subgroup of odd order of $GL_4(2)$, it follows that $|N_N(D,b_D):C_N(D)| = |N_G(D,B_D):C_G(D)|$. The same is true for each block of $C_G(D)$ covering $b_D$. It follows that $N_G(D)$ possesses $|G:N|$ blocks covering $(b_D)^{N_N(D)}$, and so $G$ possesses $|G:N|$ blocks covering $b$ by~\cite{hk85}.

Now consider the case $C_G(D)$ is not in $N_G(D,b_D)$. Then $B_D$ covers $|G:N|$ many $C_G(D)$-conjugates of $b_D$. We have $|N_G(D,B_D):N_G(D,b_D)|=|G:N|$, so $|N_N(D,b_D):C_N(D)| = |N_G(D,B_D):C_G(D)|$. The same is true for each block of $C_N(D)$ covered by $B_D$. It follows that $N_N(D)$ possesses $|G:N|$ blocks covered by $(B_D)^{N_G(D)}$, and so $N$ possesses $|G:N|$ blocks covered by $b$ by~\cite{hk85}, contradicting the $G$-stability of $b$.
\end{proof}

When it comes to reducing the desired result to the consideration of quasi-simple groups, we will be unable to rule out the case that we have a quasi-simple normal subgroup of index $3$. We must consider this situation in more detail. This is the object of the next results.

For a group $G$ and a subgroup $X$ of $G$, denote by $\Aut(G)_X$ the subgroup of $\Aut(G)$ which leaves $X$ invariant.
Denote by $\overline{ \Aut(G)_X} $ the image of $\Aut(G)_X$ in $\Aut(X)$ through the restriction map.
 Denote by $\Aut_G(X)$ the subgroup of $\Aut(X)$ of automorphisms induced by conjugation by elements of $G$. So $\Aut_G(X) $ is naturally isomorphic to $N_G(X)/C_G(X)$.

\begin{proposition}\label{class-elem-auto} Let $q$ be an odd prime power and $n$ a natural number.
Let $G=SL_n(q)/Z_0$ (respectively $SU_n(q)/Z_0$), where $Z_0 $ is a central subgroup of $SL_n(q) $ (respectively $SU_n(q)$) and suppose that $Z(G)=O_{2'}(G) $. Let $P \leq G $ be a defect group of a $2$-block of $G$
and suppose that $P$ is elementary abelian of order $2^t \geq 8$. Set $u=t+2$ if $n$ is even and $u=t+1$ if $n$ is odd.
Then $\overline{\Aut (G)_P} = \Aut_G(P)$ and $\Aut_G(P)$ is a subquotient of $ S_{u}$, where $S_u$ denotes the symmetric group on $u$ letters.
\end{proposition}

\begin{proof} Let us first consider the case that $G= SL_n(q)/Z_0$. By the statements and proofs of \cite[Lemmas 12.4, 13.4]{KKL},
 $t$ is even, $ q \equiv -3\pmod{8}$. Moreover, replacing $P$ by a $G$-conjugate if necessary, $P$ has the following structure:
Consider $GL_n(q)$ in its natural matrix representation. There is a decomposition
\[n = n_1 + \cdots + n_u, \]
into odd natural numbers $n_i $ such that denoting by $a_i $ the diagonal element of $GL_n(q)$ with entry $-1$ in positions $n_1 + \cdots +n_{i-1} + 1, \ldots , n_1 + \cdots +n_{i} $ and entry
$1$ elsewhere, and by $T_0$ the subgroup of $GL_n(q)$ generated by the elements $a_ia_j$, $1\leq i, j\leq u $, $P =(T_0Z_0)/Z_0$
(here $n_0 $ is to be taken to be $0$).

Let $\sigma \in \Aut (G)_P$. Since $PSL_n(q)$ is simple (as $n \geq 3 $), $\sigma $ lifts to an automorphism of $SL_n(q)$. We denote the lift of $\sigma $ also by $\sigma$.
Since $ T_0 $ is the unique Sylow $2$-subgroup of the inverse image of $P$ in $SL_n(q)$, $\sigma \in
\Aut (SL_n(q))_{T_0} $. Thus, in order to prove the first assertion, it suffices to prove that $ \overline {\Aut (SL_n(q))_{T_0}} = \Aut_ {SL_n(q)}(T_0)$.

Let $H$ be the subgroup of diagonal matrices of $GL_n(q)$. Let $q=p^r $, $p$ a prime and let $ \varphi : SL_n(q) \to SL_n(q)$ be the automorphism which raises every matrix entry to the $p$-th power. Let $ \tau : SL_n(q) \to SL_n(q) $ be the transpose inverse map.
By the structure of the automorphism groups of $SL_n(q)$ (see for instance \cite[Theorems 2.5.12, 2.5.14] {GLS3}),
$\Out(SL_n(q))$ is generated by the images of $\tau $, $\varphi $ and $\Aut_H(SL_n(q))$ in $\Out(SL_n(q))$
On the other hand, $\tau $, $\varphi $ and all elements of $\Aut_H(SL_n(q)) $ act as the identity on $T_0 $. Thus, $ \overline {\Aut (SL_n(q))_{T_0}} = \Aut_ {SL_n(q)}(T_0)$ and the first assertion is proved.

It has been shown above that there is a surjective homomorphism from $\Aut (SL_n(q))_{T_0} $ to $ \Aut(G)_P $ and that
 $\Aut (SL_n(q))_{T_0} = \Aut_ {SL_n(q)}(T_0)$. Thus, in order to
prove the second assertion, it suffices to prove that $ \Aut (SL_n(q))_{T_0} $ is isomorphic to a subgroup of $S_u $.

Let $V$ be an ${\mathbb F_q}$-vector space underlying the natural matrix representation of $GL_n(q)$ and for each $i$, $1\leq i \leq u $, let $V_i $ be the $-1$ eigenspace of $a_i $. So $ V = \oplus_{1\leq i\leq u } V_i $ and $dim (V_i) = n_i $, $1\leq i\leq u $. Since $T_0 $ is generated by pairs of involutions $a_ia_j $, the $-1 $ eigenspaces of elements of $T_0 $ are precisely of the form $ \oplus _{i \in I } V_i$, where $I$ ranges over subsets of even cardinality of $\{ 1, \ldots, u \} $.

Let $ g \in N_{SL_n(q)}(T_0) $, and let $i, j, k \in \{ 1,\ldots, u \}$ be pairwise distinct (this is possible since $|P| \geq 8 $ implies that $ u \geq 4 $). Since $V_i \oplus V_j $ is the $-1 $-eigenspace of $a_ia_j$, $\,^g(V_i \oplus V_j) $ is the $-1$ eigenspace of $\,^g(a_ia_j) \in T_0 $. Hence, $\,^g(V_i \oplus V_j) $ (and similarly $\,^g(V_j \oplus V_k) $) is a direct sum of an even number of the $V_i$'s. Since $ \, ^gV_i =
\, ^g( V_i\oplus V_j ) \cap \, ^g( V_i\oplus V_k) $, and since for any two subsets $J, K $ of $\{1, \ldots, u \} $, $(\oplus_{t\in J} V_t) \cap (\oplus _{t\in K} V_t)= \oplus_{t\in J\cap K} V_t$ we have that $\,^gV_i $ is also a direct sum of some of the $V_i$'s, say $\,^gV_i = \oplus_{t\in I'} V_t $. So, $V_i= \oplus_{t\in I'} \, ^{g^{-1} } V_t $. But by the same argument as before, applied to $g^{-1}$, it follows that $ I'$ consists of a single element. Hence,
for any $g \in N_{SL_n(q)}(T_0)$ and any $i$, $1\leq i \leq u$, $\,^g V_i= V_j $ for some $j$, $1\leq j \leq u $.
Further, again by considering triples of three indices $i, j, k $ one sees that $ g\in C_{SL_n(q)}(T_0)$ if and only if
$\,^gV_i = V_i $ for all $ i\in I $. Thus, $ \Aut _{SL_n(q)}(T_0) \cong N_{SL_n(q)}(T_0)/C_{SL_n(q)}(T_0)$ is isomorphic to a subgroup of $S_u$ as required.

This proves the proposition in case $G$ is a quotient of $SL_n(q)$. The case of $SU_n(q)$ is similar and we omit the details.
\end{proof}

\begin{proposition}\label{elem16-3} Let $H$ be a finite group with $Z:= Z(H) =O_{2'}(H)$ cyclic of odd order and $G$ a quasi-simple group such that $Z \leq G \unlhd H $ and $H/Z \leq \Aut(G/Z)$, $[H:G]=3 $. Let $A$ be a $2$-block of $H$ and $B$ an $H$-stable $2$-block of $G$ covered by $A$.
Suppose that the defect groups of $A$ and $B$ are elementary abelian of order $16$. Then the inertial index of $A$ is not $15$.
\end{proposition}

\begin{proof} Note that since $H/Z \leq \Aut(G/Z) $, $Z= Z(G)$. Let $ D \cong C_2 \times C_2 \times C_2 \times C_2 \leq G$ be a defect group of $A$ and of $B$. Note that by Lemma \ref{inertialquotients}, the inertial quotient of $B$ contains a subgroup isomorphic to $C_5 $ and hence $\Aut_G(D)$ contains a subgroup of order $5$.

Again, we go through the various possibilities for $G $ and $\bar G=G/Z$. If $\bar G$ is an alternating or sporadic group, then $G$ does not have a block with defect group $D$.
If $\bar G$ is a finite group of Lie type in characteristic $2$, not isomorphic to any of $\,^2F_4(2)'$, $ B_2(2)'$ or $PSp_4(2)$ then $\bar D:= D Z(G)/Z(G) \cong D $ is a
Sylow $2$-subgroup of $\bar G$, hence $\bar G =PSL_2(2^4)$. But $ \bar G \ne SL_2 (2^4 ) $ as $\Out (SL_2(2^4)) $ is a $2$-group.
The cases that $\bar G $ is isomorphic to one of $\,^2F_4(2)'$, $ B_2(2)'$ or $PSp_4(2)$ can be handled as in the proof of Theorem \ref{ab-2-str}, as can the case that $G$ is an exceptional extension of $\bar G$.

Suppose that $\bar G$ is a finite group of Lie type in odd characteristic and that $G$ is a non-exceptional extension of $\bar G$. By Proposition \ref{2abelian-levi}, if $\bar G$ is a symplectic or orthogonal group, then $B$ is nilpotent, a contradiction. If $\bar G$ is a projective special linear or
unitary group, then since neither $S_5 $ nor $S_6$ contain an element whose order is divisible by $15$, we get a contradiction by Lemma \ref{class-elem-auto}.

Thus $\bar G$ is of exceptional type. Since $B$ is not nilpotent and $\bar G$ is not of type $A_n$, by Proposition~\ref{nilp-cov-nilp}, $B$ is Morita equivalent to a block $C$ of a finite group $L$ with defect group $D'\cong D $ such that $D'$ is a product of two factors of rank $2$ each of which is invariant in $N_L(D')$.
It has been shown above that $\Aut_G(D)$ contains a subgroup of order $5$. In particular, $\Aut_G(D)$ does not leave invariant any proper non-trivial direct factor of $D$ and by \cite{ks13}, $B$ is not Morita equivalent to a block whose inertial quotient does leave a non-trivial factor of $D$ invariant.
 \end{proof}

\begin{theorem} \label{eldef16}
Let $B$ be a block of a finite group $G$ with elementary abelian defect group $D$ of order $16$ and inertial quotient $C_{15}$. Then $k(B)=k_0(B)=16$ and $l(B)=15$.
\end{theorem}

\begin{proof}
It is clear from Brauer's second main theorem that $k(B)-l(B)=1$.

Let $B$ be a counterexample to $k(B)=16$ with $(|G:Z(G)|,|G|)$ minimised in the lexicographic ordering. Hence by~\cite{ks13} $k(B)=8$.

By the first Fong reduction and minimality, $B$ is quasi-primitive. By the second Fong reduction and minimality, $O_{2'}(G)$ is cyclic and central in $G$.

Suppose that $N \lhd G$ with $N \cap D \neq 1$. Since $N_G(D)$ acts transitively on the non-trivial elements of $D$, it follows that $D \leq N$. In particular, if $N=O_2(G)$, then $D=O_2(G)$. But then $k(B)=16$, contradicting our choice of $B$. Hence $O_2(G)=1$.

By Proposition \ref{kp}, if $N \lhd G$ with $N \cap D = 1$, then by minimality $N \leq Z(G)$.

If $N=O^2(G) \neq G$, then since $B$ is quasi-primitive there is a unique block $b$ of $N$ covered by $B$ (and $B$ is the unique block of $G$ covering $b$). But then $D \cap N$ is a defect group of $b$ and $DN/N \in \Syl_2(G/N)$, so by the above $D \cap N = 1$. But then $N \leq Z(G)$, a contradiction. Hence $O^2(G)=G$.

Let $N$ be a normal subgroup of $G$ minimal subject to strictly containing $Z(G)$. Then $D \leq N$, and $N=F^*(G)$. Let $L_1,\ldots,L_t$ denote the components of $G$. We have seen that these are permuted transitively by $G$, and $L:=L_1 * \cdots * L_t \lhd G$. Then $D \leq L$. Let $B_L$ be the unique block of $L$ covered by $B$, and let $B_i$ be the unique block of $L_i$ covered by $B_L$ ($i=1,\ldots,t$). Then $B_L$ has defect group $D \cap L$, and $B_i$ has defect group $D \cap L_i$ ($i=1,\ldots,t$). Thus $D \cap L = (D \cap L_1) \times \cdots \times (D \cap L_t)$, where $D \cap L_1, \ldots,D \cap L_t$ are conjugate in $G$ (since $B_1,\ldots,B_t$ are). This implies that $t \in \{1,2,4\}$. However the existence of an element of order $15$ transitively permuting the non-trivial elements of $D$ then forces $t=1$. Hence $L$ is quasi-simple and $F^*(G)=Z(G)L$.

By the Schreier conjecture, $G/F^*(G)$ is solvable. Suppose $N \leq G$ with $|G:N|=w$, where $w$ is prime. Then $w$ is odd. Let $b$ be the unique block of $N$ covered by $B$. Note that since $C_{15}$ is a maximal odd order subgroup of $GL_4(2)$, the inertial quotient of $b$ must be a subgroup of $C_{15}$, and further if $w >5$, then it must be $C_{15}$, in which case $k(b)=16$ by minimality. By~\cite{ks13}, in any case $k(b)=8$ or $16$.

Suppose first that $B$ is the unique block of $G$ covering $b$. Consider the action of $G$ on the irreducible characters of $b$. If $w \geq 11$, there is a fixed $\theta \in \Irr(b)$, and by Clifford theory $k(B)\geq w$, a contradiction. Suppose $w=7$. Since $k(b)=16$, there are at least two fixed $\theta \in \Irr(b)$, and so $k(B)>14$, a contradiction.

Suppose $w=5$ and that $k(b)=16$. We must turn to Brauer characters to obtain a contradiction. We have $l(b)=15$. Then either $l(B)>25$ or $l(B) =3$, according to whether there are fixed points or not. In either case we have a contradiction.

Suppose that $w=5$ and that $k(b)=8$. Then $l(b)=3$, from which it follows that $l(B)=15$, a contradiction.

The case $w=3$ is ruled out by Proposition \ref{elem16-3}.

Suppose that $B$ is not the unique block of $G$ covering $b$. Then by Lemma \ref{inertialquotients}, $B$ is one of $w$ blocks covering $b$, and $b$ has inertial index $15$. Hence $k(b)=16$ by minimality. Write $t$ for the number of $G$-orbits of $\Irr(b)$. Then by Clifford theory the number of irreducible characters in blocks of $G$ covering $b$ is $(16-tw)w+t$. On the other hand, there are $w$ blocks covering $b$, each with $8$ irreducible characters. Hence $t=\frac{8w}{w^2-1}$. But $\frac{8w}{w^2-1}$ is only an integer when $w=3$. Again, the case $w=3$ is ruled out by Proposition \ref{elem16-3}.

We have shown that $G$ is quasi-simple with centre of odd order. Then we are in one of the cases (i), (iii) or (iv) of Theorem \ref{ab-2-str}. In case (i), $G \cong PSL_2(16)$, where it is indeed the case that the principal block has inertial index $15$. Checking using~\cite{gap}, we see that $k(B)=16$ in this case. In case (iii), by Proposition \ref{puig}, $B$ is Morita equivalent to a block of $N_G(D)$, and we see that $k(B)=16$. In case (iv) we again have $k(B)=16$, since Morita equivalence preserves the defect of a block and a block with Klein $4$-defect group has four irreducible characters, and we are done.
\end{proof}

%%%%%%%%%%%%%%%%%%%%%%%%%%%%%%%%%%%%%%%%%%%%%%%%%%%%%%%%%%%%%%%%%%%

\section{Donovan's conjecture for groups of the form $C_{2^m} \times C_{2^m} \times C_2$ for $m \geq 3$}
\label{homocyclicplusabit}
In this section,   by an $\ell $-block of a finite group $G$, we will mean  a  block of $kG$.
\begin{theorem} \label{2m2m2}
Let $D=C_{2^m} \times C_{2^m} \times C_2$, where $m \geq 3$. Then Donovan's conjecture holds for $D$.
\end{theorem}

\begin{proof}
It suffices to consider blocks $B$ of groups $G$ satisfying the conditions in Proposition \ref{duvel}. Let $D$ be a defect group for $B$, and write $D=P \times Q$, where $P \cong C_{2^m} \times C_{2^m}$ and $Q \cong C_2$.

We show that we may further assume that $O^2(G)=G$. Suppose that $N \lhd G$ with $|G:N|=2$. Since $B$ is quasi-primitive it follows that $G=ND$. If $N \cap D \cong P$, then $G=N \rtimes Q$ and by~\cite{kk96}, $B$ is Morita equivalent to a block of $N \times Q$ with defect group $D$. It then follows by Theorem \ref{homocyclictheorem} that there are only two Morita equivalence classes of such blocks. Hence $N \cap D \cong C_{2^m} \times C_{2^{m-1}} \times C_2$. Since $m \geq 3$, it follows that $\Aut(N \cap D)$ is a $2$-group, so that $B$ covers a nilpotent block of $N$, and we are done in this case by Proposition \ref{kp}. Hence we may suppose that $O^2(G)=G$. It follows by Schreier's conjecture that we may assume that $G/Z(G)$ is a direct product of simple groups. Further, we may take $O_2(Z(G))=1$ or $D = P \times O_2(Z(G))$, as otherwise $B$ would correspond to a nilpotent block of $G/O_2(Z(G))$, and would itself be nilpotent by Proposition \ref{watanabe}. Further, it is clear that we may also take $G$ to 
have a single component, i.e., that $G$ is quasi-simple. The result follows by Theorem~\ref{ab-2-str}.
\end{proof}

\begin{center} ACKNOWLEDGEMENTS \end{center}

The first two authors thank the Friedrich-Schiller-Universit\"at Jena for their hospitality during their visits. The fourth author is supported by the German Academic Exchange Service (DAAD), the German Research Foundation (DFG) and the Carl Zeiss Foundation.

\end{document}